\documentclass[11pt, a4paper]{article}
\usepackage{}
\usepackage{amsthm}
\usepackage{mathrsfs}
\usepackage{amsmath,amssymb,latexsym,color}
\usepackage[colorlinks,
linkcolor=blue,
anchorcolor=green,
citecolor=magenta
]{hyperref}
\usepackage{graphicx}
\usepackage{tikz}
\usepackage{caption}
\usepackage{subcaption}
\usetikzlibrary{calc}
\usepackage{CJK}

\usepackage{authblk}

\long\def\delete#1{}

\usepackage{color}

\definecolor{Blue}{rgb}{0,0,1}
\definecolor{Red}{rgb}{1,0,0}
\definecolor{DarkGreen}{rgb}{0,0.6,0}
\definecolor{DarkYellow}{rgb}{1,1,0.2}
\definecolor{DarkPurple}{rgb}{.6,0,1}

\usepackage{xcolor}
\usepackage[normalem]{ulem}

\usepackage{enumerate}
\usepackage{cleveref}
\crefformat{section}{\S#2#1#3}
\crefformat{subsection}{\S#2#1#3}
\crefformat{subsubsection}{\S#2#1#3}
\crefrangeformat{section}{\S\S#3#1#4 to~#5#2#6}
\crefmultiformat{section}{\S\S#2#1#3}{ and~#2#1#3}{, #2#1#3}{ and~#2#1#3}
\def\ma{\mathcal{A}}
\def\mb{\mathcal{B}}
\def\mc{\mathcal{C}}
\def\md{\mathcal{D}}

\def\mf{\mathcal{F}}
\def\mg{\mathcal{G}}
\def\mh{\mathcal{H}}
\def\mi{\mathcal{I}}

\def\mm{\mathcal{M}}

\def\mt{\mathcal{T}}

\def\bs{\setminus}

\numberwithin{equation}{section}
\numberwithin{figure}{section}

\newtheorem{cl}{Claim}
\newtheorem{as}{Assumption}
\newtheorem{thm}{Theorem}[section]
\newtheorem{ex}[thm]{Example}
\newtheorem{lem}[thm]{Lemma}

\newtheorem{cor}[thm]{Corollary}

\begin{document}

\setcounter{page}{1}
\renewcommand{\thefootnote}{}
\newcommand{\remark}{\vspace{2ex}\noindent{\bf Remark.\quad}}
\renewcommand{\abovewithdelims}[2]{%
	\genfrac{[}{]}{0pt}{}{#1}{#2}}


\def\qed{\hfill$\Box$\vspace{11pt}}

\title {\bf $s$-almost cross-$t$-intersecting families for finite sets}

\author{Dehai Liu\thanks{E-mail: \texttt{liudehai@mail.bnu.edu.cn}}\   \textsuperscript{a}}
\author{Kaishun Wang\thanks{ E-mail: \texttt{wangks@bnu.edu.cn}}\ \textsuperscript{a}}
\author{Tian Yao\thanks{Corresponding author. E-mail: \texttt{tyao@hist.edu.cn}}\ \textsuperscript{b}}
\affil{ \textsuperscript{a} Laboratory of Mathematics and Complex Systems (Ministry of Education), School of
	Mathematical Sciences, Beijing Normal University, Beijing 100875, China}

\affil{ \textsuperscript{b} School of Mathematical Sciences, Henan Institute of Science and Technology, Xinxiang 453003, China}
\date{}

\openup 0.5\jot
\maketitle

\begin{abstract}
	
  Two  families  $\mathcal{F}$ and $\mathcal{G}$   of $k$-subsets of an $n$-set   are called $s$-almost cross-$t$-intersecting if  each member in $\mathcal{F}$ (resp. $\mathcal{G}$) is $t$-disjoint with at most $s$ members in $\mathcal{G}$ (resp. $\mathcal{F}$). In this paper, we  characterize  the  $s$-almost cross-$t$-intersecting families with the maximum product of their sizes. Furthermore, we provide  a  corresponding stability result after studying  the $s$-almost cross-$t$-intersecting families which are not cross‑$t$-intersecting.

	\vspace{2mm}
	
	\
	
	
\end{abstract}

\section{Introduction}

Let $n$, $k$,  $t$  and $s$ be positive integers with $n\geq  k\geq t$.  Write $[n]=\left\{1,2,\ldots,n\right\}$.  For a set $X$, denote the family of all $k$-subsets of $X$   by $\binom{X}{k}$. Two sets $F$ and $G$ are said to be  \textit{$t$-intersecting}  if $\left|F\cap G\right|\geq t$, and \textit{$t$-disjoint} otherwise.

A family $\mf\subseteq \binom{[n]}{k}$ is called \textit{$t$-intersecting} if any two members in $\mf$ are $t$-intersecting.  The  structure of  maximum-sized $t$-intersecting families  was completely 
determined by the  Erd\H{o}s-Ko-Rado theorem \cite{2406295,MR0140419, MR0519277, MR0771733},  and  a stability result of this theorem  was  known as the  Hilton-Milner theorem \cite{2504097, 2504095,2504096,2407162}.  Analogously, two families $\mf,\mg\subseteq \binom{[n]}{k}$ are  said to be \textit{cross-$t$-intersecting} if  any member in $\mathcal{F}$ and  any member in $\mathcal{G}$ are $t$-intersecting. There is a vast literature on studying the   maximum product of  sizes of cross-$t$-intersecting families. We  refer the readers to  \cite{2504098, 25040910, 25040911,2506256,25040912} for details.

These problems have attracted widespread attention.   Several generalizations have been made by relaxing the intersection property  in various ways.   One extension of $t$-intersecting families is to consider a family $\mf\subseteq \binom{[n]}{k}$ where
each member in $\mf$ is  $t$-disjoint with at most $s$ members in $\mf$. Such a family is said to be 
\textit{$s$-almost $t$-intersecting}, which has been studied in \cite{2410232,2406096,2406095,2502161,2412121},  and  its $q$-analog can be found in \cite{2506253,2506252}.
Similarly,  one may locally weaken the  intersection property between two families.   
We say that two  families $\mf,\mg\subseteq \binom{[n]}{k}$  are  \textit{$s$-almost cross-$t$-intersecting} if each member in $\mathcal{F}$ (resp. $\mathcal{G}$) is $t$-disjoint with at most $s$ members in $\mathcal{G}$ (resp. $\mathcal{F}$).

Let $\mf$ and $\mg$ be $s$-almost cross-$t$-intersecting families.  In \cite{25040913}, Gerbner et al. pointed out  the maximum value of $\left|\mf\right|\left|\mg\right|$ for $t=1$.  Our first result characterizes  families $\mf$ and $\mg$ with the maximum $\left|\mf\right|\left|\mg\right|$ for $t\geq 1$.

\begin{thm}\label{1}
	Let $n$, $k$, $t$ and $s$ be positive integers with $k\geq t+1$ and $n\geq (t+1)(2(k-t+1)^{2}+7s)$. Suppose that $\mf,\mg\subseteq \binom{[n]}{k}$ are $s$-almost cross-$t$-intersecting families. If $\left|\mf\right|\left|\mg\right|$ is maximum, then there exists $W\in\binom{[n]}{t}$ such that $\mf=\mg=\left\{ H\in\binom{[n]}{k}: W\subseteq H\right\}$ .
\end{thm}

  In light of Theorem \ref{1}, it is natural to study   $\mf$ and $\mg$ that maximize  $\left|\mf\right|\left|\mg\right|$, subject to the condition  $\left|\bigcap_{H\in\mf\cup\mg}H\right|<t$. Notice that   the extremal structure has been provided in \cite{2410172} when $\mf$ and $\mg$ are cross-$t$-intersecting.  In this paper, we address the complementary case.

To present our following results, we introduce some notations. Suppose that $W$, $X$ and $Y$ are subsets of $[n]$ with $W\subseteq X$. Write
\begin{equation*}
	\begin{aligned}
		\mh_{1}(X,W;k)&=\left\{ F\in\binom{X}{k}: W\subseteq F\right\},\\
		\mh_{2}(X,W;k)&=\left\{ F\in\binom{[n]}{k}: F\cap X=W\right\},\\
		\mm_{1}(Y;k,t)&=\left\{F\in\binom{[n]}{k}:\left|F\cap Y\right|\geq t\right\},\\
		\mm_{2}(X,W;t)&=\left\{ F\in\binom{X}{t+1}: \left|F\cap W\right|=t-1\right\}. 		
	\end{aligned}
\end{equation*}

\begin{thm}\label{2}
	Let $n$, $k$, $t$ and $s$ be positive integers with $k\geq t+2$ and $n\geq  (t+1)^{2}(2(k-t+1)^{2}+7s)$. Suppose that $\mf,\mg\subseteq \binom{[n]}{k}$ are $s$-almost cross-$t$-intersecting  but not cross-$t$-intersecting families with $\left|\mf\right|\leq \left|\mg\right|$.   If  $\left|\mf\right|\left|\mg\right|$ is maximum,  then there exist		 $X\in\binom{[n]}{k+1}$ and $W\in\binom{X}{t}$ such that
	$$\mf=\mh_{1}([n], W;k)\bs \ma\ \ \textnormal{and}\ \ \mg=\mh_{1}([n], W;k)\cup \mb,$$
	where $\ma$ is an  $\left(\binom{n-k-1}{k-t}-s\right)$-subset of $\mh_{2}(X,W;k)$ and $\mb$ is a $\min\left\{t,s\right\}$-subset of $\binom{X}{k}\bs\mh_{1}(X,W;k)$.
\end{thm}

\begin{thm}\label{3}
	Let $n$, $t$ and $s$ be positive integers with $n\geq 5s(t+1)^{2}$. Suppose that $\mf, \mg \subseteq \binom{[n]}{t+1}$ are $s$-almost cross-$t$-intersecting  but not cross-$t$-intersecting families with $\left|\mf\right|\leq  \left|\mg\right|$, and   $\left|\mf\right|\left|\mg\right|$ is maximum. 
	\begin{enumerate}[\normalfont(i)]
		\item   If $t\geq s+2$, then there exists $Y\in\binom{[n]}{t+1}$ such that 
		$$\mf=\left\{Y\right\}\ \ \textnormal{and} \ \ \mg=\mm_{1}(Y;t+1,t)\cup \mc,$$
		where $\mc$ is an $s$-subset of $\binom{[n]}{t+1}\bs\mm_{1}(Y;t+1,t)$.
		\item  If  $t\leq s+1$, then there exist $Z\in\binom{[n]}{t+s+2}$ and $W\in\binom{Z}{t}$ such that 
		$$\mf=\mh_{1}(Z,W;t+1)\ \ \textnormal{and} \ \ \mg=\mh_{1}([n],W;t+1)\cup \md,$$
		where $\md$ is an $(s+2)$-subset of $\mm_{2}(Z,W;t)$ for which each element of $Z\bs W$ is contained in exactly two members in $\md$.
	\end{enumerate}
\end{thm}

By combining \cite[Theorem 1.2]{2410172}, Theorems \ref{2} and \ref{3},  we may determine the structure of $s$-almost cross-$t$-intersecting families $\mf$ and $\mg$ with the maximum $\left|\mf\right|\left|\mg\right|$ under the condition  $\left|\bigcap_{H\in\mf\cup\mg}H\right|<t$. See Corollaries \ref{5} and \ref{6}.

The rest of this paper is organized as follows. In Section \ref{2506041}, we show some auxiliary results.   Theorems \ref{1}, \ref{2} and \ref{3}  are proved in Sections \ref{2503171}, \ref{2503172} and  \ref{2503173} respectively.  In Section  \ref{2505231}, we derive a stability result of Theorem \ref{1}.   To ensure a smoother reading experience, some   inequalities needed in this paper are left in Section \ref{2504132}.

\section{Preliminaries}\label{2506041}
In this section, we  show some auxiliary results  which are used in the proof of main theorems. For a family $\mf$ and a set $H$, write
$$\mf_{H}=\left\{F\in\mf: H\subseteq F\right\},\ \ \md_{\mf}(H;t)=\left\{ F\in\mf: \left|F\cap H\right|<t\right\}.$$

\begin{lem}\label{2502162}
	Let $n$, $k$, $t$ and $s$ be positive integers with $k\geq t+1$ and $n\geq 2k$. Suppose that $\mf\subseteq \binom{[n]}{k}$  and $G\in \binom{[n]}{k}$  satisfy  $\left|\md_{\mf}(G;t)\right|\leq s$. If $H$ is a non-empty subset of $[n]$ with $\left|H\cap G\right|<t$, then  there exists a	subset $R$ of $[n]$ with $H\subsetneq R$ such that
	$$\left| \mf_{H}\right|\leq (k-t+1)^{\left| R\right|-\left|H\right|}\left| \mf_{R}\right|+s.$$
\end{lem}

\begin{proof}
	
	If $\mf=\md_{\mf}(G;t)$, then the required result holds. Next assume $\mf\bs\md_{\mf}(G;t)\neq \emptyset$.
	
	Let $\mi=\mf\bs \md_{\mf}(G;t)$. 	
	Then $\left|G\cap I\right|\geq t$ for any $I\in\mi$. By \cite[Lemma 2.7]{2410172}, there exists a $\left(\left|H\right|+t-\left|H\cap G\right|\right)$-subset $R$ with $H\subsetneq R$ such that $$\left|\mi_{H}\right|\leq \binom{k-\left|H\cap G\right|}{t-\left|H\cap G\right|}\left| \mi_{R}\right|\leq \binom{k-\left|H\cap G\right|}{t-\left|H\cap G\right|}\left| \mf_{R}\right|.$$
	It follows from $\mi_{H}\subseteq \mf_{H}$ and $\left| \mf_{H}\bs\mi_{H}\right|\leq \left|\mf\bs \mi\right|\leq s$ that
	\begin{equation}\label{2502164}
		\left|\mf_{H}\right|\leq \binom{k-\left|H\cap G\right|}{t-\left|H\cap G\right|}\left| \mf_{R}\right|+s.
	\end{equation}

	Observe that
	$$\frac{k-\left|H\cap G\right|}{t-\left|H\cap G\right|}\leq \frac{k-\left| H\cap G\right|-1}{t-\left|H\cap G\right|-1}\leq \cdots \leq \frac{k-t+1}{1}.$$
	Then $\binom{k-\left|H\cap G\right|}{t-\left|H\cap G\right|}\leq (k-t+1)^{t-\left|H\cap G\right|}=(k-t+1)^{\left|R\right|-\left|H\right|}$. This together with (\ref{2502164}) yields  the desired result.
\end{proof}

Suppose that $\mf\subseteq\binom{[n]}{k}$. 
A subset $T$  of $[n]$ is called a \textit{$t$-cover} of $\mf$ if $\left|T\cap F\right|\geq t$ for any $F\in \mf$.
The \textit{$t$-covering number} $\tau_{t}(\mf)$ of $\mf$ is the minimum size of a $t$-cover of $\mf$. The family of all $t$-covers with size $\tau_{t}(\mf)$ is denoted by $\mt_{t}(\mf)$.

From now on, we assume that $s$-almost cross-$t$-intersecting families are non-empty, since those families we are concerned with  must be non-empty.

\begin{lem}\label{2502165}
	Let $n$, $k$, $t$ and $s$ be positive integers with $k\geq t+1$ and $n\geq \left(t+1\right)\left(k-t+1\right)^{2}$. Suppose that $\mf,\mg \subseteq \binom{[n]}{k}$ are $s$-almost cross-$t$-intersecting families with $\tau_{t}(\mg)\leq k$. If  $H$ is a non-empty subset of $[n]$ with $ \left|H\right|\leq \tau_{t}(\mg)$, then
	$$\left|\mf_{H}\right|\leq (k-t+1)^{\tau_{t}(\mg)-\left|H\right|}\binom{n-\tau_{t}(\mg)}{k-\tau_{t}(\mg)}+\sum_{i=0}^{\tau_{t}(\mg)-\left|H\right|-1}s(k-t+1)^{i}.$$
\end{lem}
\begin{proof}
	If $\left|H\right|=\tau_{t}(\mg)$, then the desired result follows from $\left|\mf_{H}\right|\leq \binom{n-\left|H\right|}{k-\left|H\right|}=\binom{n-\tau_{t}(\mg)}{k-\tau_{t}(\mg)}$. Next assume $\left|H\right|<\tau_{t}(\mg)$.

	From $\left|H\right|<\tau_{t}(\mg)$, we know $\left|H\cap G_{1}\right|<t$ for some $G_{1}\in\mg$. Note that  $\left|\md_{\mf}(G_{1};t)\right|\leq s$. By Lemma \ref{2502162}, there exists a subset $H_{2}$ of $[n]$ with $H\subsetneq H_{2}$ such that
	$$\left| \mf_{H}\right|\leq (k-t+1)^{\left| H_{2}\right|-\left|H\right|}\left| \mf_{H_{2}}\right|+s.$$
	Repeated the process above, we finally get an ascending chain of subsets $H=:H_{1}\subsetneq H_{2}\subsetneq \cdots\subsetneq H_{u}$ with $\left|H_{u-1}\right|<\tau_{t}(\mg)\leq \left|H_{u}\right|$ such that
	$$\left| \mf_{H_{i}}\right|\leq (k-t+1)^{\left|H_{i+1}\right|-\left|H_{i}\right|}\left|\mf_{H_{i+1}}\right|+s,\ \ 1\leq i\leq u-1.$$
	This implies 
	\begin{equation*}
		\left|\mf_{H}\right|\leq (k-t+1)^{\left|H_{u}\right|-\left|H\right|}\left| \mf_{H_{u}}\right|+\sum_{i=1}^{u-1}s(k-t+1)^{\left|H_{i}\right|-\left|H\right|}.
	\end{equation*}
	It follows from $\sum_{i=1}^{u-1}s(k-t+1)^{\left|H_{i}\right|-\left|H\right|}\leq \sum_{i=0}^{\tau_{t}(\mg)-\left|H\right|-1}s(k-t+1)^{i}$ that 
	\begin{equation}\label{2502166}
		\left|\mf_{H}\right|\leq (k-t+1)^{\left|H_{u}\right|-\left|H\right|}\left| \mf_{H_{u}}\right|+\sum_{i=0}^{\tau_{t}(\mg)-\left|H\right|-1}s(k-t+1)^{i}.
	\end{equation}

	Since $t\leq \tau_{t}(\mg)\leq \left|H_{u}\right|$, by Lemma \ref{2410176}, we have 
	$$(k-t+1)^{\left|H_{u}\right|-\left|H\right|}\left| \mf_{H_{u}}\right|\leq (k-t+1)^{\left|H_{u}\right|-\left|H\right|} \binom{n-\left|H_{u}\right|}{k-\left|H_{u}\right|}\leq (k-t+1)^{\tau_{t}(\mg)-\left|H\right|}\binom{n-\tau_{t}(\mg)}{k-\tau_{t}(\mg)}.$$
	This together with (\ref{2502166}) yields the desired result.
\end{proof}

In the subsequent two lemmas, we establish some upper bounds for the product of sizes of $s$-almost cross-$t$-intersecting families in terms of their $t$-covering numbers. For positive integers $n$, $k$, $t$, $s$ and $x$, write 
\begin{equation}\label{2505251}
	f_{1}(n,k,t,s,x)=(k-t+1)^{x-t}\binom{x}{t}\binom{n-x}{k-x}+\sum_{i=0}^{x-t-1}s(k-t+1)^{i}\binom{x}{t}. 
\end{equation}

\begin{lem}\label{2503081}
	Let $n$, $k$, $t$ and $s$ be positive integers with $k\geq t+1$ and $n\geq \left(t+1\right)\left(k-t+1\right)^{2}$. Suppose that $\mf,\mg \subseteq \binom{[n]}{k}$ are $s$-almost cross-$t$-intersecting families. If  $\tau_{t}(\mf)\leq k$  and $\tau_{t}(\mg)\leq k$, then
	$$\left|\mf\right|\left|\mg\right|\leq f_{1}(n,k,t,s,\tau_{t}(\mf))f_{1}(n,k,t,s,\tau_{t}(\mg)).$$
\end{lem}
\begin{proof}
	Let $T_{f}\in\mt_{t}(\mf)$ and $T_{g}\in\mt_{t}(\mg)$. Then 
	$\mf=\bigcup_{H\in\binom{T_{f}}{t}}\mf_{H}$ and $\mg=\bigcup_{H\in\binom{T_{g}}{t}}\mg_{H}$.
	It follows  that
	$$\left|\mf\right|\leq \binom{\tau_{t}(\mf)}{t}\cdot\max_{H\in\binom{T_{f}}{t}}\left|\mf_{H}\right|, \ \ \left|\mg\right|\leq \binom{\tau_{t}(\mg)}{t}\cdot\max_{H\in\binom{T_{g}}{t}}\left|\mg_{H}\right|.$$
	This together with Lemma \ref{2502165} yields $\left|\mf\right|\left|\mg\right|\leq f_{1}(n,k,t,s,\tau_{t}(\mf))f_{1}(n,k,t,s,\tau_{t}(\mg))$.
\end{proof}

\begin{lem}\label{2502184}
	Let $n$, $k$, $t$ and $s$ be positive integers with $k\geq t+1$ and $n\geq 2k$. Suppose that $\mf,\mg \subseteq \binom{[n]}{k}$ are $s$-almost cross-$t$-intersecting families. If  $\tau_{t}(\mg)\geq  k+1$, then 
	$$\left|\mf\right|\left|\mg\right|\leq s\binom{k}{t}\binom{2k-2t+2}{k-t+1}\binom{n-t}{k-t}+s^{2}\binom{2k-2t+2}{k-t+1}.$$
\end{lem}
\begin{proof}

	Let $V_{1}=\mf$. Choose $F_{i}$, $G_{i}$ and $V_{i+1}$ by repeating the following steps:
	$$F_{i}\in V_{i}, \ \ G_{i}\in\md_{\mg}(F_{i};t),\ \ V_{i+1}=V_{i}\bs \md_{\mf}(G_{i};t).$$
	Since $\left|V_{i+1}\right|<\left|V_{i}\right|$, there exists a positive integer $m$ such that $V_{m+1}=\emptyset$. Finally, we get two  
	sequences of subsets $F_{1}, F_{2},\ldots, F_{m}$ and $G_{1}, G_{2},\ldots, G_{m}$ satisfy
	\begin{itemize}
		\item[(a)] $\left|F_{i}\cap G_{i}\right|<t$ for any $1\leq i\leq m$.
		\item[(b)] $\left|F_{i}\cap G_{j}\right|\geq t$ for any $1\leq j< i\leq m$.
		\item[(c)] $\mf=\cup_{i=1}^{m}\md_{\mf}(G_{i};t)$.
	\end{itemize}

	According to \cite[Theorem 6]{2406291}, two sequences of subsets $F_{1},F_{2},\ldots,F_{m}$ and $G_{1}, G_{2},\ldots,G_{m}$ satisfying  (a) and (b) must have $m\leq \binom{2k-2t+2}{k-t+1}$.	Since $\mf,\mg$ are $s$-almost cross-$t$-intersecting,
	by (c), we have
	$$\left|\mf\right|\leq \sum_{i=1}^{m}\left|\md_{\mf}(G_{i};t)\right|\leq s \binom{2k-2t+2}{k-t+1}.$$

	Let $F\in\mf$. Then $F$ is a $t$-cover of $\mg\bs \md_{\mg}(F;t)$. We further derive  $\mg\bs \md_{\mg}(F;t)=\bigcup_{H\in\binom{F}{t}}\mg_{H}$, which implies
	\begin{equation*}
		\left|\mg\right|= \left|\mg\bs \md_{\mg}(F;t)\right|+\left|\md_{\mg}(F;t) \right|\leq \binom{k}{t}\binom{n-t}{k-t}+s.
	\end{equation*}
	This together with the upper bound for $\left|\mf\right|$ yields the desired result. 
\end{proof}

The $s$-almost cross-$t$-intersecting families $\mf$ and $\mg$ are said to be \textit{maximal} if $\mf^{\prime}=\mf$ and $\mg^{\prime}=\mg$ for any $s$-almost cross-$t$-intersecting families $\mf^{\prime}$ and $\mg^{\prime}$ with $\mf\subseteq \mf^{\prime}$ and $\mg\subseteq \mg^{\prime}$. 
We proceed by proving a property of the  minimum $t$-covers of such families.

\begin{lem}\label{2502182}
	Let $n$, $k$, $t$ and $s$ be positive integers with $k\geq t+1$ and $n\geq 2k$. Suppose that $\mf,\mg\subseteq \binom{[n]}{k}$ are maximal $s$-almost cross-$t$-intersecting families.  If $\tau_{t}(\mf)\leq k$ and $\tau_{t}(\mg)\leq k$, then $\mt_{t}(\mf)$ and $\mt_{t}(\mg)$ are cross-$t$-intersecting.
\end{lem}
\begin{proof}
	Let $T_{f}\in\mt_{t}(\mf)$ and $T_{g}\in\mt_{t}(\mg)$. According to $n-\left|T_{f}\cup T_{g}\right|\geq  k-\tau_{t}(\mg)$, there exists $U\in\binom{[n]}{k}$ such that $T_{g}\subseteq U$ and $T_{f}\cap T_{g}=T_{f}\cap U$.	
	Since $\mf$, $\mg$ are maximal and $T_{g}\subseteq U$, we know $U\in \mf$. It follows from $T_{f}\in \mt_{t}(\mf)$ and $T_{f}\cap T_{g}=T_{f}\cap U$ that $\left|T_{f}\cap T_{g}\right| \geq t$.
\end{proof}

\section{Proof of Theorem \ref{1}}\label{2503171}

Let $n$, $k$, $t$ and $s$ be positive integers with $k\geq t+1$ and $n\geq (t+1)(2(k-t+1)^{2}+7s)$. Suppose that $\mf,\mg\subseteq \binom{[n]}{k}$ are $s$-almost cross-$t$-intersecting families with the maximum product of their sizes. 
Since  $\mh_{1}([n],[t];k)$ and $\mh_{1}([n],[t];k)$ are  $s$-almost cross-$t$-intersecting, we have 
\begin{equation}\label{2506251}
	\left|\mf\right|\left|\mg\right|\geq \binom{n-t}{k-t}^{2}.
\end{equation}

Next, we show that $(\tau_{t}(\mf),\tau_{t}(\mg))=(t,t)$. Suppose for contradiction that $(\tau_{t}(\mf),\tau_{t}(\mg))\neq (t,t)$. If $\tau_{t}(\mathcal{F})\leq k$ and $\tau_{t}(\mathcal{G})\leq k$, then by   Lemma \ref{2503081}, we have
$$\left|\mf\right|\left|\mg\right|\leq f_{1}(n,k,t,s,\tau_{t}(\mf))f_{1}(n,k,t,s,\tau_{t}(\mg)).$$
This together with $\left(\tau_{t}(\mf),\tau_{t}(\mg)\right)\neq (t,t)$ and Lemma \ref{2502181} (i) yields
$$\left|\mf\right|\left|\mg\right|\leq f_{1}(n,k,t,s,t)f_{1}(n,k,t,s,t+1)<f_{1}(n,k,t,s,t)^{2}=\binom{n-t}{k-t}^{2},$$
a contradiction to (\ref{2506251}). If $\tau_{t}(\mathcal{F})\geq  k+1$ or $\tau_{t}(\mathcal{G})\geq k+1$, then by Lemma \ref{2502184}, we obtain
$$\left|\mf\right|\left|\mg\right|\leq s\binom{k}{t}\binom{2k-2t+2}{k-t+1}\binom{n-t}{k-t}+s^{2}\binom{2k-2t+2}{k-t+1}.$$
It follows from Lemma \ref{2502181} (i) and (ii) that
\begin{equation*}
	\left|\mf\right|\left|\mg\right|\leq \frac{6}{7}f_{1}(n,k,t,s,t)\binom{n-t}{k-t}+\frac{6s}{7}f_{1}(n,k,t,s,t)=\frac{6}{7}\binom{n-t}{k-t}^{2}+\frac{6s}{7}\binom{n-t}{k-t}.
\end{equation*}
This combining with 
$6s<7s(t+1)\leq n-t\leq \binom{n-t}{k-t}$
yields $\left|\mf\right|\left|\mg\right|<\binom{n-t}{k-t}^{2}$, a contradiction to (\ref{2506251}). Therefore, we have $(\tau_{t}(\mf),\tau_{t}(\mg))=(t,t)$.

Pick $T_{f}\in\mt_{t}(\mf)$ and $T_{g}\in\mt_{t}(\mg)$. Then 
$\mf\subseteq \mh_{1}([n],T_{f};k)$ and $\mg\subseteq \mh_{1}([n],T_{g};k)$.
Since $\left|\mf\right|\left|\mg\right|$ is maximum, we know  $\mf$ and $\mg$ are maximal. It follows from Lemma \ref{2502182}  that $T_{f}=T_{g}=:W$. By (\ref{2506251}), we conclude
$\mf=\mg= \mh_{1}([n],W;k)$. \hfill $\square$

\section{Proof of Theorem \ref{2}}\label{2503172}

In this section, we first derive a lower bound for the product of sizes of the families  in the assumption of Theorem \ref{2}, via the following example. This  bound serves to  rule out certain candidate extremal structures.

\begin{ex}\label{2503041}
	Let $n$, $k$, $t$, $s$ be positive integers with  $\binom{n-k-1}{k-t}\geq s$, and $X\in\binom{[n]}{k+1}$,  $W\in\binom{X}{t}$. The families 
	$$\mh_{1}([n],W;k)\bs \ma\ \ \textnormal{and}\ \ \mh_{1}([n], W;k)\cup \mb, $$
	where $\ma$ is an $\left(\binom{n-k-1}{k-t}-s\right)$-subset of $\mh_{2}(X,W;k)$  and $\mb$ is a $\min\left\{t,s\right\}$-subset of $\binom{X}{k}\bs\mh_{1}(X,W;k)$, are $s$-almost cross-$t$-intersecting but not cross-$t$-intersecting. 
\end{ex}

The  product of  sizes of families in Example \ref{2503041} is 
\begin{equation}\label{2505252}
	\left(\binom{n-t}{k-t}-\binom{n-k-1}{k-t}+s\right)\left( \binom{n-t}{k-t}+\min\left\{t,s\right\}\right)=:g_{1}(n,k,t,s).
\end{equation} 

\begin{lem}\label{2503082}
	Let $n$, $k$, $t$ and $s$ be positive integers with $k\geq t+2$ and $n\geq (t+1)^{2}(2(k-t+1)^{2}+7s)$. Suppose that $\mf,\mg\subseteq \binom{[n]}{k}$ are $s$-almost cross-$t$-intersecting families. If  $\left(\tau_{t}(\mf),\tau_{t}(\mg)\right)\notin\left\{(t,t), (t,t+1),(t+1,t) \right\}$, then $\left|\mf\right|\left|\mg\right|< g_{1}(n,k,t,s)$.
\end{lem}
\begin{proof}
	Based on the $t$-covering numbers of $\mf$ and $\mg$,	we divide our proof into following cases.

	\medskip
	\noindent \textbf{Case 1.} $\tau_{t}(\mathcal{F})\leq k$ and $\tau_{t}(\mathcal{G})\leq k$. 
	\medskip
	
	Since $\left(\tau_{t}(\mf),\tau_{t}(\mg)\right)\notin\left\{(t,t), (t,t+1),(t+1,t) \right\}$,  one of the following  holds: $\tau_{t}(\mf)\geq t+1$ and $\tau_{t}(\mg)\geq t+1$, or $\tau_{t}(\mf)=t $ and $\tau_{t}(\mg)\geq t+2$, or 
	$\tau_{t}(\mf)\geq t+2 $ and $\tau_{t}(\mg)= t$.

	\medskip
	\noindent \textbf{Case 1.1.} $\tau_{t}(\mf)\geq t+1$ and $\tau_{t}(\mg)\geq t+1$. 
	\medskip
	
	By Lemma \ref{2503081}  and Lemma \ref{2502181} (i), we have 	$\left|\mf\right|\left|\mg\right|\leq f_{1}(n,k,t,s,t+1)^{2}$.
	This together with Lemma \ref{2503091} (i) yields the desired result.

	\medskip
	\noindent \textbf{Case 1.2.} $\tau_{t}(\mf)=t $ and $\tau_{t}(\mg)\geq t+2$, or $\tau_{t}(\mf)\geq t+2 $ and $\tau_{t}(\mg)= t$.
	\medskip

	From  Lemma \ref{2503081} and Lemma \ref{2502181} (i), we derive $\left|\mf\right|\left|\mg\right|\leq f_{1}(n,k,t,s,t)f_{1}(n,k,t,s,t+2)$. 
	It follows from Lemma \ref{2503091} (ii) that
    the required result holds.

	\medskip
	\noindent \textbf{Case 2.} $\tau_{t}(\mathcal{F})\geq  k+1$ or $\tau_{t}(\mathcal{G})\geq k+1$. 
	\medskip
	
	 By Lemma \ref{2502184}, we have
	$\left|\mf\right|\left|\mg\right|\leq s\binom{k}{t}\binom{2k-2t+2}{k-t+1}\binom{n-t}{k-t}+s^{2}\binom{2k-2t+2}{k-t+1}.$
	It follows from Lemma \ref{2502181} (i) and (ii) that 
	\begin{equation*}
		\begin{aligned}
			\left|\mf\right|\left|\mg\right|&\leq \frac{6}{7(t+1)}f_{1}(n,k,t,s,t+1)\binom{n-t}{k-t}+\frac{6s}{7(t+1)}f_{1}(n,k,t,s,t+1)\\
			&= \frac{6}{7(t+1)}f_{1}(n,k,t,s,t+1)\left(\binom{n-t}{k-t}+s\right).
		\end{aligned}
	\end{equation*}
This combining with Lemma \ref{2503091} (iii) implies the required result.
\end{proof}

In order to prove Theorem \ref{2}, we now analyze $s$-almost cross-$t$-intersecting families $\mf$ and $\mg$ with $(\tau_{t}(\mathcal{F}), \tau_{t}(\mathcal{G})) \in \{ (t,t+1), (t+1,t) \}$.  Specifically, we consider such families under some necessary conditions imposed by  the extremal structure.

\begin{as}\label{2506061}
Let $n$, $k$, $t$ and $s$ be positive integers with $k\geq t+2$ and $n\geq (t+1)^{2}(2(k-t+1)^{2}+7s)$. Suppose that $\mf,\mg\subseteq \binom{[n]}{k}$ are maximal $s$-almost cross-$t$-intersecting families with $\left(\tau_{t}(\mf),\tau_{t}(\mg)\right)=(t,t+1)$, but they are not cross-$t$-intersecting. Set $W=\bigcup_{T_{f}\in\mt_{t}(\mf)}T_{f}$ and $X=\bigcup_{T_{g}\in \mt_{t}(\mg)}T_{g}$. 
\end{as}

\begin{lem}\label{2503083}
	Let $n$, $k$, $t$, $s$, $\mf$, $\mg$, $W$ and $X$ be as in Assumption \ref{2506061}.  Then the following hold.
	\begin{enumerate}[\normalfont(i)]
			\item  $W\in\binom{X}{t}$.
			\item For each $G\in \mg\bs\mg_{W}$, we have $X\subseteq W\cup G\in\binom{[n]}{k+1}$.
			\item For each $H\in\mh_{1}([n], W;k)$ with $\md_{\mg}(H;t)\neq \emptyset$, we have $H\cap X=W$. 
		\end{enumerate}
\end{lem}
\begin{proof} 
	For each $F\in\mf$, by $\tau_{t}(\mf)=t$, we obtain $W\subseteq F$.
	From Lemma \ref{2502182}, we know that $\mt_{t}(\mf)$ and $\mt_{t}(\mg)$ are cross-$t$-intersecting.
	It follows from $\tau_{t}(\mf)=t$ that $W\subseteq T_{g}$ for any $T_{g}\in\mg$.

	(i)   Let $F\in\mf$ with $\md_{\mg}(F;t)\neq \emptyset$, and $T_{g}\in \mt_{t}(\mg)$. It follows that 
	$t\geq \left|F\cap T_{g}\right|\geq \left|W \right|\geq t$, which implies $W=F\cap T_{g}\in\binom{T_{g}}{t}\subseteq \binom{X}{t}$. 

	(ii)  From (i), we know $\left|W\right|=t$.  For any $G\in\mg\bs\mg_{W}$ and $T_{g}\in\mt_{g}$, since $W\subseteq T_{g}$  and $\left|T_{g}\cap G\right|\geq t$,
	we have  
	$\left|W\cap G\right|=t-1$ and $\left|T_{g}\cap (W\cup G)\right|\geq t+1$, implying that $\left|W\cup G\right|=k+1$ and $T_{g}\subseteq W\cup G$. Hence $X\subseteq W\cup G\in\binom{[n]}{k+1}$.

	(iii)  Let $H\in\mh_{1}([n],W;k)$ with $\md_{\mg}(H;t)\neq \emptyset$, and $G\in\md_{\mg}(H;t)$. Then $G\in \mg\bs\mg_{W}$. By (ii), we have  $X\subseteq W\cup G\in\binom{[n]}{k+1}$, which implies 
	\begin{equation*}
			\begin{aligned}
					t-1\geq \left|H\cap G\right|=\left|H\cap G\cap \left( W\cup G\right)\right|\geq \left|H\cap\left( W\cup G\right)\right|+\left|G\right|-\left| W\cup G\right|\geq \left|H\cap X\right|-1.
				\end{aligned}
		\end{equation*}
	This together with $W\subseteq H\cap X$ and $\left|W\right|=t$ yields the desired result. 
\end{proof}

For positive integers $n$, $k$, $t$, $s$ and $x$, write
\begin{align}
	f_{2}(n,k,t,s)&=\binom{n-t-1}{k-t-1}+(k-t)(k-t+1)\binom{n-t-2}{k-t-2}+s(k-t+1), \label{2506122} \\ 
	f_{3}(n,k,t,s,x)&=\binom{n-t}{k-t}-\binom{n-x}{k-t}+ (k-x+1)^{2}\binom{n-t-2}{k-t-2}+2s. \label{2505253}
\end{align}

\begin{lem}\label{2504066}
	Let $n$, $k$, $t$, $s$, $\mf$, $\mg$, $W$ and $X$ be as in Assumption \ref{2506061}.  Then the following hold.
	\begin{enumerate}[\normalfont(i)]
			\item $\left|\mg\right|\leq \binom{n-t}{k-t}+t(k-t)\binom{n-\left|X\right|}{k-\left|X\right|}+s$.
			\item If $\left|\mt_{t}(\mg)\right|=1$, then $\left|\mf\right|\left|\mg\right|<g_{1}(n,k,t,s)$. 
		\end{enumerate}
\end{lem}
\begin{proof}
	Let $A\in\mf$ and $B\in\mg$ with $\left|A\cap B\right|<t$. Then $B\in\mg\bs\mg_{W}$ due to $\md_{\mf}(B;t)\neq \emptyset$, and $W\in\binom{X}{t}$ due to Lemma \ref{2503083} (i). 

	(i) If $ \mg= \mg_{W}\cup \md_{\mg}(A;t)$, then $\left|\mg \right|\leq \binom{n-t}{k-t}+s$, the required result holds. Next we assume that $\mg\bs\left( \mg_{W}\cup \md_{\mg}(A;t)\right)\neq \emptyset$.

	For each $G\in \mg\bs\left( \mg_{W}\cup \md_{\mg}(A;t)\right)$, 
	by $W\subseteq X$ and  Lemma \ref{2503083} (ii), we get 
	$W\cup G\subseteq X\cup G \subseteq W\cup G$ and $\left|W\cup G\right|=k+1$, implying that $\left|G\cap X\right|=\left|X\right|-1$ and $\left|G\cap W\right|=t-1$. It follows from Lemma \ref{2503083} (iii) that 
	$$\left| G\cap \left(  A\cup X\right)\right|=\left| G\cap A\right|+\left| G\cap X\right|-\left|G\cap W\right|\geq \left|X\right|,$$
	which implies $\left|X\right|\leq k$.
	Therefore, we have
	$$\mg\bs\left( \mg_{W}\cup \md_{\mg}(A;t)\right)\subseteq \bigcup_{H\in\binom{A\cup X}{\left|X\right|},\ \left|H\cap X\right|=\left|X\right|-1,\ W\nsubseteq H} \mg_{H}.$$
	Note that 
	$$\left| \left\{H\in\binom{A\cup X}{\left|X\right|}:\ \left|H\cap X\right|=\left|X\right|-1,\ W\nsubseteq H \right\}\right|=t(k-t).$$
	We further conclude 
	$\left|\mg\bs\left( \mg_{W}\cup \md_{\mg}(A;t)\right)\right|\leq t(k-t)\binom{n-\left|X\right|}{k-\left|X\right|}$. This yields   $$\left|\mg\right|\leq \binom{n-t}{k-t}+ t\left(k-t\right)\binom{n-\left|X\right|}{k-\left|X\right|}+s$$
	due to $\mg=\mg_{W}\cup\left(\mg\bs\left( \mg_{W}\cup \md_{\mg}(A;t)\right)\right) \cup \md_{\mg}(A;t)$.

	(ii) Since $\mt_{t}(\mg)$ has the unique member, we have $\mt_{t}(\mg)=\left\{X\right\}$.  From Lemma \ref{2503083} (ii), we obtain $X\subseteq W\cup B\in\binom{[n]}{k+1}$.

	For each $F\in\mf\bs\md_{\mf}(B;t)$, since $W\subseteq F$ and $\left|W\cap B\right|=t-1$, we have $\left|F\cap \left(W\cup B\right)\right|\geq t+1$. Hence 
	\begin{equation}\label{2504061}
			\mf\bs\md_{\mf}(B;t)\subseteq \mf_{X}\cup \left(\bigcup_{H\in\binom{W\cup B}{t+1},\ W\subseteq F,\ H\neq X}\mf_{H}\right).
		\end{equation}
	Let $H\in\binom{W\cup B}{t+1}$ with $W\subseteq H$ and $H\neq X$. There exists $G\in\mg$ such that $\left|H\cap G\right|<t$ due to $\mt_{t}(\mg)=\left\{X\right\}$. It follows from Lemma \ref{2502162} that $\left|\mf_{H}\right|\leq (k-t+1)^{\left|R\right|-\left|H\right|}\left|\mf_{R}\right|+s$ for some set $R$ with $H\subsetneq R$.  This together with Lemma \ref{2410176} yields $\left|\mf_{H}\right|\leq (k-t+1)\binom{n-t-2}{k-t-2}+s$. By (\ref{2504061}), we get
	$\left|\mf\right|\leq f_{2}(n,k,t,s)$.

	From (i) and $\left|X\right|=t+1$, we have $\left|\mg\right|\leq \binom{n-t}{k-t}+t(k-t)\binom{n-t-1}{k-t-1}+s$. 
	It follows from  $\left|\mf\right|\leq f_{2}(n,k,t,s)$ and  Lemma \ref{2503091} (iv) that $\left|\mf\right|\left|\mg\right|<g_{1}(n,k,t,s)$.
\end{proof}

\begin{lem}\label{2506062}
		Let $n$, $k$, $t$, $s$, $\mf$ and $\mg$ be as in Assumption \ref{2506061}. If  $\left|\mf\right|\left|\mg\right|\geq g_{1}(n,k,t,s)$, then there exist $X\in\binom{[n]}{k+1}$ and $W\in\binom{X}{t}$ such that 
		$$\mf=\mh_{1}([n], W;k)\bs \ma\ \ \textnormal{and}\ \ \mg=\mh_{1}([n], W;k)\cup \mb,$$
		where $\ma$ is an $\left(\binom{n-k-1}{k-t}-s\right)$-subset of $\mh_{2}(X,W;k)$ and $\mb$ is a $\min\left\{t,s\right\}$-subset of $\binom{X}{k}\bs\mh_{1}(X,W;k)$.
\end{lem}
\begin{proof}
Set $W=\bigcup_{T_{f}\in\mt_{t}(\mf)}T_{f}$ and $X=\bigcup_{T_{g}\in \mt_{t}(\mg)}T_{g}$. 	Let $A\in\mf$ and $B\in\mg$ with $\left|A\cap B\right|<t$. Then $B\in\mg\bs\mg_{W}$ due to $\md_{\mf}(B;t)\neq \emptyset$, and $W\in\binom{X}{t}$ due to Lemma \ref{2503083} (i).  By $\left|\mf\right|\left|\mg\right|\geq g_{1}(n,k,t,s)$ and Lemma \ref{2504066} (ii), we get $\left|X\right|\geq t+2$.

	\begin{cl}\label{2503103}
		$X=W\cup G\in\binom{[n]}{k+1}$ for any $G\in\mg\bs\mg_{W}$. 
	\end{cl}
	\begin{proof}
		From Lemma \ref{2503083} (ii), we obtain $X\subseteq W\cup G\in\binom{[n]}{k+1}$ for any $G\in\mg\bs \mg_{W}$.

		We proceed by  showing  $W\cup B=W\cup G$ 	for any $G\in\mg\bs\mg_{W}$. 
		Suppose for contradiction that there exists $C\in \mg\bs\mg_{W}$ such that $W\cup B\neq W\cup C$. Let $M=\left(W\cup B\right)\cap \left(W\cup C\right)$.  By  $X\subseteq M\subsetneq W\cup B\in\binom{[n]}{k+1}$, we know  $t+2\leq \left|M\right|\leq k$. Set
		$$\mh=\left\{(H_{1}, H_{2})\in\binom{W\cup B}{t+1}\times \binom{W\cup C}{t+1}: W\subseteq H_{1}\nsubseteq M,\ W\subseteq H_{2}\nsubseteq M\right\}.$$

		For each $F\in\mf\bs\left( \md_{\mf}(B;t)\cup \md_{\mf}(C;t)\right)$, since $W\subseteq F$ and $\left|W\cap B\right|=\left|W\cap C\right|=t-1$, we have $\left|F\cap \left(W\cup B\right)\right|\geq t+1$ and $\left|F\cap \left( W\cup C\right)\right|\geq t+1$. Therefore, we get
		$$\mf\subseteq \left\{F\in\binom{[n]}{k}: W\subseteq F, \ \left|F\cap M\right|\geq t+1\right\}\cup \left(\bigcup_{(H_{1},H_{2})\in \mh}\mf_{H_{1}\cup H_{2}}\right)\cup \md_{\mf}(B;t)\cup \md_{\mf}(C;t).$$	
		Note that $\left|H_{1}\cup H_{2}\right|=t+2$	 for any $(H_{1}, H_{2})\in\mh$. Then 
		$$\left| \mf\right|\leq \binom{n-t}{k-t}-\binom{n-\left|M\right|}{k-t}+ (k-\left|M\right|+1)^{2}\binom{n-t-2}{k-t-2}+2s=f_{3}(n,k,t,s,\left|M\right|).$$
		This together with Lemma \ref{2503042} (i) yields	 	$\left|\mf\right|\leq f_{3}(n,k,t,s,k)$. From Lemma \ref{2504066} (i) and $\left|X\right|\geq t+2$, we obtain
		$$\left|\mf\right|\left|\mg\right|\leq f_{3}(n,k,t,s,k)\left( \binom{n-t}{k-t}+t(k-t)\binom{n-t-2}{k-t-2}+s\right).$$
		It follows from Lemma \ref{2503042} (ii) that $\left|\mf\right|\left|\mg\right|<g_{1}(n,k,t,s)$, a contradiction to $\left|\mf\right|\left|\mg\right|\geq g_{1}(n,k,t,s)$. Hence $W\cup B=W\cup G$ 	for any $G\in\mg\bs\mg_{W}$.
		
		Let $G\in \mg$ and  $T\in\binom{W\cup B}{t+1}$ with $W\subseteq T$. If $G\in \mg_{W}$, then $\left|T\cap G\right|\geq t$. If $G\in\mg\bs\mg_{W}$, then  by $\left|W\cup G\right|=k+1$, we have $\left|W\cap G\right|=t-1$. It follows from $W\cup G=W\cup B$ that
		$$\left| T\cap G\right|=\left|T\cap G\right|+\left|T\cap W\right|-\left|W\cap G\right|-1=\left| T\cap \left( W\cup G\right)\right|-1=\left|T\cap \left(W\cup B\right)\right|-1=t.$$
		Therefore, we have  $T\in \mt_{t}(\mg)$. We further conclude $\left\{ T\in \binom{W\cup B}{t+1}: W\subseteq T\right\}\subseteq \mt_{t}(\mg)$, which implies $W\cup B\subseteq X$. This together with  $X\subseteq W\cup B$ yields $X=W\cup B=W\cup G\in\binom{[n]}{k+1}$ for any $G\in\mg\bs\mg_{W}$.
	\end{proof}

	Note that $\mf\subseteq \mh_{1}([n],W;k)$.  To  characterize $\mf$ and $\mg$, we first describe $\mh_{1}([n],W;k)\bs \mf$ and $\mg\bs\mg_{W}$.

	By Claim \ref{2503103}, we know $X=W\cup B\in\binom{[n]}{k+1}$, which implies $\md_{\mh_{1}([n],W;k)}(B;t)=\mh_{2}(X,W;k)$.
	It follows from $n-k-1\geq k-t+s$ that  $$\left|\md_{\mh_{1}([n],W;k)}(B;t)\right|=\binom{n-k-1}{k-t}\geq \binom{k-t+s}{k-t}=\binom{k-t+s}{s}>s.$$ Hence  $\mf\subsetneq \mh_{1}([n],W;k)$.   
	Pick $H\in \mh_{1}([n],W;k)\bs \mf$. By the maximality of $\mf$ and $\mg$, we know $\md_{\mg}(H;t)\neq \emptyset$. From Lemma \ref{2503083} (iii), we obtain $H\in \mh_{2}(X, W;k)$, which implies 
	\begin{equation}\label{2505221}
		\mh_{1}([n],W;k)\bs \mf\subseteq \mh_{2}(X, W;k).
	\end{equation}
	It follows that  $$\mh_{2}(X,W;k)\bs\mf \subseteq \mh_{1}([n],W;k)\bs \mf \subseteq \mh_{2}(X,W;k)\bs\mf.$$
	Therefore, we have $\mh_{1}([n],W;k)\bs \mf=\mh_{2}(X,W;k)\bs\mf$. This combining with $\md_{\mf}(B;t)=\mh_{2}(X, W;k)\cap \mf$ yields
	\begin{equation}\label{2505222}
		\left|\mh_{1}([n],W;k)\bs \mf \right|= \left|\mh_{2}(X,W;k) \right|-\left|\md_{\mf}(B;t) \right|\geq \binom{n-k-1}{k-t}-s.
	\end{equation}

	Choose $G\in\mg\bs\mg_{W}$. By Claim \ref{2503103}, we have $G\subseteq W\cup G=X$. Hence 
	\begin{equation}\label{2505223}
		\mg\bs\mg_{W}\subseteq \binom{X}{k}\bs \mh_{1}(X,W;k).
	\end{equation}
	From Lemma \ref{2503083} (iii), we know $A\cap X=W$, which implies  $\left|A\cap G\right|=\left| A\cap X \cap G\right|=\left| W\cap G\right|<t$.
	It follows that 
	\begin{equation}\label{2505225}
		\left|\mg\bs\mg_{W}\right|\leq \min\left\{ \left|\binom{X}{k}\bs \mh_{1}(X,W;k)\right|, \left|\md_{\mg}(A;t)\right|\right\}\leq\min\left\{t,s\right\}.
	\end{equation}

	Now, according to $\left|\mf\right|\left|\mg\right|\geq g_{1}(n,k,t,s)$,  we have 
	\begin{equation*}
		\begin{aligned}
		 \left(\left|\mh_{1}([n],W;k)\right|-\left|\mh_{1}([n],W;k)\bs \mf \right|\right)\left( \left|\mg_{W}\right|+\left|\mg\bs\mg_{W}\right|\right)
			=\left|\mf\right|\left| \mg\right|\geq g_{1}(n,k,t,s).
		\end{aligned}
	\end{equation*}
	This together with (\ref{2505221})--(\ref{2505225}) implies that $\mh_{1}([n],W;k)\bs\mf$ is an $\left(\binom{n-k-1}{k-t}-s \right)$-subset of $\mh_{2}(X,W;k)$, $\mg_{W}=\mh_{1}([n],W;k)$ and $\mg\bs\mg_{W}$ is a $\min\left\{t,s\right\}$-subset of $\binom{X}{k}\bs \mh_{1}(X,W;k)$. Note that $\mf=\mh_{1}([n],W;k)\bs\left(\mh_{1}([n],W;k)\bs\mf\right)$ and $\mg=\mg_{W}\cup\left(\mg\bs\mg_{W} \right)$. Then the desired result holds.
\end{proof}

Now, we are  ready to prove Theorem \ref{2}.

\begin{proof}[\textnormal{\textbf{Proof of Theorem \ref{2}}}]
	 Note that $\mf,\mg$ are maximal, and they are not cross-$t$-intersecting. By Lemma \ref{2502182}, we know $\left(\tau_{t}(\mf),\tau_{t}(\mg)\right)\neq (t,t)$. 
	 Since $\left|\mf\right|\left|\mg\right|$ is maximum,  by Example \ref{2503041}, we have  	$\left|\mf\right|\left|\mg\right|\geq g_{1}(n,k,t,s)$.
	 It follows from $\left|\mf\right|\leq \left|\mg\right|$ and  	 
     Lemma \ref{2506062} that  $\left(\tau_{t}(\mf),\tau_{t}(\mg)\right)\neq(t+1,t)$. We further conclude that $(\tau_{t}(\mf),\tau_{t}(\mg))=(t,t+1)$ due to Lemma \ref{2503082}. 
	  Applying Lemma \ref{2506062} again, we know Theorem \ref{2} holds.
\end{proof}

\section{Proof of Theorem \ref{3}}\label{2503173}

In this section, we focus on the  $s$-almost cross-$t$-intersecting families consisting of $(t+1)$-subsets of $[n]$,  and prove Theorem \ref{3}. We start our investigation with two examples.

\begin{ex}\label{2503261}
	Let $n$, $t$, $s$ be positive integers with $\sum_{i=0}^{t-1}\binom{t+1}{i}\binom{n-t-1}{t-i+1}\geq s$, and $Y\in\binom{[n]}{t+1}$. 
	The families 
	$$\left\{Y\right\}\ \ \textnormal{and} \ \ \mm_{1}(Y;t+1,t)\cup \mc, $$
	where $\mc$ is an $s$-subset of $\binom{[n]}{t+1}\bs\mm_{1}(Y;t+1,t)$, are $s$-almost cross-$t$-intersecting but not cross-$t$-intersecting. 
\end{ex}

\begin{ex}\label{2503262}
	Let $n$, $t$, $s$ be positive integers with $n\geq t+s+2$, and $Z\in\binom{[n]}{t+s+2}$, $W\in\binom{Z}{t}$.
	The families 
	$$\mh_{1}(Z,W;t+1)\ \ \textnormal{and} \ \ \mh_{1}([n],W;t+1)\cup \md, $$
	where $\md$ is an $(s+2)$-subset of $\mm_{2}(Z,W;t)$ for which each element of $Z\bs W$ is contained in exactly two members in $\md$, are $s$-almost cross-$t$-intersecting but not cross-$t$-intersecting.
\end{ex}

The products of  sizes of families in Examples \ref{2503261} and  \ref{2503262} are
\begin{align}
		(t+1)(n-t)+s-t&=:g_{2}(n,t,s),\label{2505255}\\    
		(s+2)(n-t)+(s+2)^{2}&=:g_{3}(n,t,s), \label{2505262}
\end{align}
respectively. Then $g_{2}(n,t,s)$ and $g_{3}(n,t,s)$ are  lower bounds for the product of sizes of families  in the assumption of  Theorem \ref{3}. By applying $g_{3}(n,t,s)$, we can determine some necessary conditions that the extremal structure should satisfy.

\begin{lem}\label{2503175}
	Let $n$, $t$ and $s$ be positive integers with $n\geq 5s(t+1)^{2}$. Suppose that $\mf,\mg\subseteq \binom{[n]}{t+1}$ are $s$-almost cross-$t$-intersecting families.  If $\tau_{t}(\mf)\geq t+1$ and $\left|\mf\right|\leq \left|\mg\right|$,
	then $\left|\mf\right|\left|\mg\right|<g_{3}(n,t,s)$. 
\end{lem}

\begin{proof}
    Let $F\in\mf$.  For each $H\in \binom{F}{t}$, since $\tau_{t}(\mf)\geq t+1$, we know $\left|F^{\prime}\cap H\right|<t$ for some $F^{\prime}\in\mf$. By Lemma \ref{2502162}, there exists a subset $R$ of $[n]$ with $H\subsetneq R$ such that  $\left|\mg_{H}\right|\leq 2^{\left|R\right|-t}\cdot\left|\mg_{R}\right|+s$. It follows from Lemma \ref{2410176} that $\left|\mg_{H}\right|\leq s+2$. 
    
    From $\mg=\md_{\mg}(F;t)\cup \left( \bigcup_{H\in \binom{F}{t}}\mg_{H}\right)$, we obtain $\left| \mg\right|\leq (t+1)(s+2)+s$. 	
	This combining with $\left|\mf\right|\leq \left|\mg\right|$ and Lemma \ref{2503174} (ii) implies the required result.
\end{proof}

\begin{lem}\label{2503181}
	Let $n$, $t$ and $s$ be positive integers with $n\geq 5s(t+1)^{2}$. Suppose that $\mf,\mg\subseteq \binom{[n]}{t+1}$ are $s$-almost cross-$t$-intersecting families. If $\tau_{t}(\mf)=t$ and $2\leq \left|\mf\right|\leq s+1$, then $\left|\mf\right|\left|\mg\right|< g_{3}(n,t,s)$.
\end{lem}

\begin{proof} 
	Let $W$ be a $t$-cover of $\mf$ with size $t$, and $F_{1}, F_{2}$ be distinct members in $\mf$. 
	For each $G\in\binom{[n]}{t+1}$ with $\left|G\cap F_{1}\right|\geq t$ and $\left|G\cap F_{2}\right|\geq t$, we know
	$$t+1\geq \left|G\cap \left( F_{1}\cup F_{2}\right)\right|=\left|G\cap F_{1}\right|+\left|G\cap F_{2}\right|-\left|G\cap W\right|\geq 2t-\left|G\cap W\right|, $$
	which implies either
	$W\subseteq G$, or $\left|G\cap W\right|=t-1$ and $G\subseteq F_{1}\cup F_{2}$. Hence
	\begin{equation*}
		\mg\bs\left(\md_{\mg}(F_{1};t)\cup \md_{\mg}(F_{2};t)\right)\subseteq \mh_{1}([n],W;t+1)\cup \mm_{2}(F_{1}\cup F_{2},W;t).
	\end{equation*}
	We further conclude that 
	\begin{equation}\label{2504021}
		\left|\mg\bs\left(\md_{\mg}(F_{1};t)\cup \md_{\mg}(F_{2};t)\right)\right|\leq n.
	\end{equation}
	
	\medskip
	\noindent \textbf{Case 1.} $\left|\mf\right|=2$. 
	\medskip
	
	Since $\left|\md_{\mg}(F_{1};t)\cup \md_{\mg}(F_{2};t)\right|\leq 2s$,
	by (\ref{2504021}), we get $\left|\mg\right|\leq n+2s$. It follows that $\left|\mf\right|\left|\mg\right|\leq 2n+4s$. This yields
	$$g_{3}(n,t,s)-\left|\mf\right|\left|\mg\right|\geq sn-\left(s+2\right)t+s^{2}+4\geq 5s^{2}(t+1)^{2}-(s+2)t >0.$$

	\medskip
	\noindent \textbf{Case 2.} $\left|\mf\right|\geq 3$. 
	\medskip

	In this case, we have $s\geq 2$. Let $\mg^{\prime}=\bigcup_{F\in\mf}\md_{\mg}(F;t)$. Next, we show 
	\begin{equation}\label{2504022}
		\left|\mg^{\prime}\right|\leq \frac{s\left|\mf\right|}{\left|\mf\right|-2}. 
	\end{equation}
	If $\mg^{\prime}=\emptyset$, then (\ref{2504022}) holds.  Now we  assume that $\mg^{\prime}\neq \emptyset$. For each 
	$G^{\prime}\in \mg^{\prime}$, we claim  $\left|\md_{\mf}(G^{\prime}; t)\right|\geq \left|\mf\right|-2$. Otherwise, there are three distinct members $F_{1}^{\prime},F_{2}^{\prime}, F_{3}^{\prime}\in\mf$ such that $\left|G^{\prime}\cap F_{i}^{\prime}\right|\geq t$ for any $i\in[3]$. Since $G^{\prime}\in\mg^{\prime}$ and  $W\subseteq F$ for any $F\in\mf$, we have $\left|G^{\prime}\cap W\right|\leq t-1$, implying that
	\begin{equation*}
		\begin{aligned}
			\left|G^{\prime}\cap \left( F_{1}^{\prime}\cup F_{2}^{\prime}\cup F_{3}^{\prime}\right)\right|=\left|G^{\prime}\cap F_{1}^{\prime}\right|+\left|G^{\prime}\cap F_{2}^{\prime}\right|+\left|G^{\prime}\cap F_{3}^{\prime}\right|-2\left|G^{\prime}\cap W\right|\geq t+2,
		\end{aligned}
	\end{equation*}
	a contradiction. By double counting, we get
	$$s\left|\mf\right|\geq \left|\left\{ \left(F, G\right)\in \mf\times \mg: \left|F\cap G\right|<t\right\}\right|\geq \left|\mg^{\prime}\right|\left( \left|\mf\right|-2\right). $$
	This yields (\ref{2504022}).
	
	Observe that $\mg=\left( \mg\bs\left(\md_{\mg}(F_{1};t)\cup \md_{\mg}(F_{2};t)\right)\right)\cup\mg^{\prime}$. It follows from  (\ref{2504021}), (\ref{2504022}) and $3\leq \left|\mf\right|\leq s+1$ that
	\begin{equation*}
		\begin{aligned}
			g_{3}(n,t,s)-\left|\mf\right|\left|\mg\right|&\geq g_{3}(n,t,s)-\left|\mf\right|\left(n+\frac{s\left|\mf\right|}{\left|\mf\right|-2}\right)\\
			&= n(s-\left|\mf\right|+2)+(s+2)^{2}-(s+2)t-\left(s\left|\mf\right|+2s+\frac{4s}{\left|\mf\right|-2} \right)\\
			&\geq 5s(t+1)^{2}+(s+2)^{2}-(s+2)t-\left(s^{2}+7s\right)\\
			&\geq \frac{5}{3}(s+2)(t+1)^{2}+4-(s+2)t-3s>0.
		\end{aligned}
	\end{equation*}
	Then the required result follows.
\end{proof}

\begin{proof}[\textnormal{\textbf{Proof of Theorem \ref{3}}}]
	Since $\left|\mf\right|\left|\mg\right|$ takes the maximum value, by Examples \ref{2503261} and \ref{2503262}, we have 
	\begin{equation}\label{2503182}
		\left|\mf\right|\left|\mg\right|\geq \max\left\{g_{2}(n,t,s), g_{3}(n,t,s)\right\}.
	\end{equation}
	It follows from Lemma \ref{2503175}  that $\tau_{t}(\mf)=t$.

	Let $W$ be a $t$-cover of $\mf$ with size $t$.   Since  $\mf$ and $\mg$ are not cross-$t$-intersecting, we know $\left|G\cap W\right|<t$  for some $G\in \mg$.   By Lemma \ref{2502162}, there exists a subset $R$ of $[n]$ with $W\subsetneq R$ such that  $\left|\mf\right|=\left|\mf_{W}\right|\leq 2^{\left|R\right|-t}\cdot\left|\mf_{R}\right|+s$. It follows from Lemma \ref{2410176} that
	\begin{equation}\label{2506031}
		\left|\mf\right|\leq s+2.
	\end{equation}
	In the following, we investigate $\mf$ and $\mg$ into two cases.

	\medskip
	\noindent \textbf{Case 1.} $\left|\mf\right|=1$. 
	\medskip
	
	Set $\mf=\left\{Y\right\}$. Then $\mg\subseteq \mm_{1}(Y;t+1,t)\cup \md_{\mg}(Y;t)$.  By (\ref{2503182}), we know $\left|\mf\right|\left|\mg\right|\geq g_{2}(n,t,s)$, implying that 
	$\mg=\mm_{1}(Y;t+1,t)\cup \md_{\mg}(Y;t)$ and   $\md_{\mg}(Y;t)$ is an $s$-subset of $\binom{[n]}{t+1}\bs\mm_{1}(Y;t+1,t)$.

	\medskip
	\noindent \textbf{Case 2.} $\left|\mf\right|\geq 2$. 
	\medskip

	From (\ref{2503182}), we know $\left|\mf\right|\left|\mg\right|\geq g_{3}(n,t,s)$. It follows from Lemma \ref{2503181} and (\ref{2506031}) that $\left|\mf\right|=s+2$. Write $Z=\bigcup_{F\in\mf}F$. Then $\left|Z\right|=t+s+2$ and $\mf=\mh_{1}(Z,W;t+1)$. 
	
	For each $G\in\mg\bs\mg_{W}$, since $\left|\md_{\mf}(G;t)\right|\leq s$, there exist distinct $F_{1}, F_{2}\in\mf$ such that $\left|G\cap F_{1}\right|\geq t$ and $\left|G\cap F_{2}\right|\geq t$. By $F_{1}\cap F_{2}=W$, we have
	$$t+1\geq \left|G\cap \left( F_{1}\cup F_{2}\right)\right|=\left|G\cap F_{1}\right|+\left| G\cap F_{2}\right|-\left| G\cap W\right|\geq 2t-\left| G\cap W\right|,$$
	implying that $\left|G\cap W\right|=t-1$  and $G\subseteq F_{1}\cup F_{2}\subseteq Z$. Therefore, we know 
	\begin{equation}\label{2505226}
		\mg\bs\mg_{W}\subseteq \mm_{2}(Z,W;t). 
	\end{equation}
	It follows from $\mf=\mh_{1}(Z,W;t+1)$ that $\left|\md_{\mf}(G;t)\right|=s$ for any $G\in\mg\bs\mg_{W}$.  Hence
	\begin{equation}\label{2504271}
		s(s+2)\geq \sum_{F\in \mf}\left|\md_{\mg\bs\mg_{W}}(F;t)\right|= \left|\left\{ (F,G)\in\mf\times \mg: \left|F\cap G\right|<t\right\}\right|=s\left|\mg\bs\mg_{W}\right|,
	\end{equation}
	which implies \begin{equation}\label{2506101}
		\left|\mg\bs\mg_{W}\right|\leq s+2.
	\end{equation} 
	
	By (\ref{2503182}), we have 
	$ (s+2)\left(\left|\mg_{W}\right|+\left|\mg\bs \mg_{W}\right| \right)=\left|\mf\right|\left|\mg\right|\geq g_{3}(n,t,s)$.
	This together with (\ref{2505226}) and (\ref{2506101}) yields $\mg_{W}=\mh_{1}([n],W;t+1)$ and $\mg\bs \mg_{W}$ is an $(s+2)$-subset of $\mm_{2}(Z,W;t)$. 	
	For each $z\in Z\bs W$, by (\ref{2504271}), we know $\left|\md_{\mg\bs\mg_{W}}(W\cup \left\{z\right\};t)\right|=s$, i.e., exactly two members in $\mg\bs\mg_{W}$ contain $z$.

	In summary, we know that $\mf$ and $\mg$ are  given in either Example \ref{2503261} or Example \ref{2503262}.  It follows from Lemma \ref{2503174} (iii) that (i) holds if $t\geq s+2$, and (ii) holds if $t\leq s+1$. 
\end{proof}

\section{Concluding remarks}\label{2505231}

Let $\mf, \mg\subseteq \binom{[n]}{k}$ be $s$-almost cross-$t$-intersecting families with $\left|\bigcap_{H\in \mf\cup \mg}H\right|<t$. As  mentioned before,  it is natural to ask:   what is  the structure of  $\mf$ and $\mg$ when  $\left|\mf\right|\left|\mg\right|$ is maximum?  We begin this section by presenting  an example   from \cite{2410172}, and then answer the above problem.
\begin{ex}\label{2504162}
	Let $n$, $k$, $t$ be positive integers with $n\geq k\geq t+1$, and $Y\in\binom{[n]}{t+1}$. The families 
	$$\mh_{1}([n],Y;k)\ \ \textnormal{and}\ \ \mm_{1}(Y;k,t)$$ 
	are  cross-$t$-intersecting, and the size of the intersection of all members in  $\mh_{1}([n],Y;k)\cup \mm_{1}(Y;k,t)$ is less than $t$.  
\end{ex}
The product of sizes of families in Example \ref{2504162} is 
\begin{equation}\label{2505261}
(t+1)\binom{n-t-1}{k-t-1}\binom{n-t-1}{k-t}+\binom{n-t-1}{k-t-1}^{2}=:g_{4}(n,k,t).
\end{equation}

\begin{cor}\label{5}
	Let $n$, $k$, $t$ and $s$ be positive integers with $k\geq t+2$ and $n\geq \max\left\{k-t,t+1\right\}\cdot (t+1) (2(k-t+1)^{2}+7s)$. Suppose that $\mf,\mg\subseteq \binom{[n]}{k}$ are $s$-almost cross-$t$-intersecting families with $\left| \bigcap_{H\in\mf\cup\mg}H\right|<t$ and  $\left|\mf\right|\leq  \left|\mg\right|$. Assume that $\left|\mf\right|\left|\mg\right|$ is maximum. 
	\begin{enumerate}[\normalfont(i)]
		\item  If $k\geq 2t+1$ or $(k,t)=(4,2)$, then     $\mf$ and $\mg$ are as in Example \ref{2503041}.
		\item  If $k\leq 2t$ and $(k,t)\neq (4,2)$,  then
		$\mf$ and $\mg$ are as in Example \ref{2504162}.
	\end{enumerate}
\end{cor}
\begin{proof}
	Since $\left|\mf\right|\left|\mg\right|$ takes the maximum value, by Examples \ref{2503041} and \ref{2504162}, we have
	\begin{equation}\label{2504163}
		\left|\mf\right|\left|\mg\right|\geq \max\left\{ g_{1}(n,k,t,s), g_{4}(n,k,t)\right\}.
	\end{equation}
	Next we divide our proof into the following two cases.

	\medskip
	\noindent \textbf{Case 1.} $k\geq 2t+1$ or $(k,t)=(4,2)$. 
	\medskip	
	
	By (\ref{2504163})  and Lemma \ref{2504143}	 (i),  we have
	$$\left|\mf\right|\left|\mg\right|\geq g_{1}(n,k,t,s) > \left( \binom{n-t}{k-t}-\binom{n-k-1}{k-t}\right)\left(\binom{n-t}{k-t}+t\right).$$
	From  \cite[Theorem 1.2]{2410172}, we know that $\mf$ and $\mg$ are not cross-$t$-intersecting. Indeed, $\mf$ and $\mg$ satisfy the assumption of Theorem \ref{2}. Then (i) holds.

	\medskip
	\noindent \textbf{Case 2.} $k\leq 2t$ and $(k,t)\neq (4,2)$. 
	\medskip

	By (\ref{2504163}) and Lemma \ref{2504143} (ii), we have
	$\left|\mf\right|\left|\mg\right|\geq g_{4}(n,k,t)> g_{1}(n,k,t,s) $. It follows from Theorem \ref{2} that $\mf$ and $\mg$ are cross-$t$-intersecting. 
	
	Since $\left|\mf\right|\left|\mg\right|$ takes the maximum value, by 	\cite[Theorem 1.2]{2410172},  there exists $Y\in\binom{[n]}{t+1}$ such that $\mf=\mh_{1}([n],Y;k)$ and $\mg=\mm_{1}(Y;k,t)$, or $\mf=\mm_{1}(Y;k,t)$ and $\mg=\mh_{1}([n],Y;k)$. This together with $\left|\mf\right|\leq \left|\mg\right|$ yields $\mf=\mh_{1}([n],Y;k)$ and $\mg=\mm_{1}(Y;k,t)$. Then (ii) holds.
\end{proof}

\begin{cor}\label{6}
		Let $n$, $t$ and $s$ be positive integers with $n\geq 5s(t+1)^{2}$. Suppose that $\mf, \mg \subseteq \binom{[n]}{t+1}$ are $s$-almost cross-$t$-intersecting families with $\left| \bigcap_{H\in\mf\cup\mg}H\right|<t$ and $\left|\mf\right|\leq \left|\mg\right|$.  Assume that  $\left|\mf\right|\left|\mg\right|$ is maximum.
		\begin{enumerate}[\normalfont(i)]
			\item If $t\geq s+2$, then $\mf$ and $\mg$ are as in Example \ref{2503261}.
			\item If $t\leq s+1$, then $\mf$ and $\mg$ are as in Example \ref{2503262}.
		\end{enumerate}
\end{cor}
\begin{proof}
By Examples \ref{2503261} and \ref{2503262}, we have $\left|\mf\right|\left|\mg\right|\geq \max\left\{g_{2}(n,t,s), g_{3}(n,t,s)\right\}$. 
It follows from	Lemma \ref{2503174} (i) and (ii) that 
 $\left|\mf\right|\left|\mg\right|>\max\left\{ 2n, (t+1)(n-t)-t\right\}$.
 By \cite[Theorem 1.2]{2410172}, we conclude that $\mf$ and $\mg$ are not cross-$t$-intersecting. Namely, $\mf$ and $\mg$ satisfy the assumption of Theorem \ref{3}. This completes the proof.
\end{proof}

We remark here that the method used in this paper can also be applied to study $s$-almost cross-$t$-intersecting families for other mathematical objects, for example, vector spaces \cite{2506221}.

\section{Inequalities concerning binomial coefficients}\label{2504132}

In this section, we prove some inequalities  concerning binomial coefficients. We first show a lemma through routine computations.

\begin{lem}\label{2410176}
	Let $n$, $k$ and $t$ be positive integers with $k\geq t+1$ and $n\geq (k-t)(k-t+1)+t$. If  $i$ and $j$ are positive integers  with $t\leq i\leq j$, then
	$$(k-t+1)^{j-i}\binom{n-j}{k-j}\leq \binom{n-i}{k-i}.$$
\end{lem}
\begin{proof}
	If $j>k$ or $i=j$, then the desired result follows.  Now we  assume that $t\leq i<j\leq k$.
	
	Note that $n\geq (k-t)(k-t+1)+t\geq k$.  Then,	for each non-negative integer $\ell\leq j-i-1$, we have 
	\begin{equation*}
		\begin{aligned}
			\frac{\binom{n-i-\ell}{k-i-\ell}}{\binom{n-i-\ell-1}{k-i-\ell-1}}= \frac{n-i-\ell}{k-i-\ell}
			\geq \frac{n-t}{k-t}
			\geq  k-t+1.
		\end{aligned}
	\end{equation*}  
	Therefore, we obtain
	\begin{equation*}
		\frac{\binom{n-i}{k-i}}{\binom{n-j}{k-j}}=\prod_{\ell=0}^{j-i-1}\frac{\binom{n-i-\ell}{k-i-\ell}}{\binom{n-i-\ell-1}{k-i-\ell-1}}\geq (k-t+1)^{j-i}.
	\end{equation*}
	This completes the proof.
\end{proof}

Recall that $f_{1}(n,k,t,s,x)$, $f_{2}(n,k,t,s)$, $f_{3}(n,k,t,s,x)$  are defined in (\ref{2505251}), (\ref{2506122}), (\ref{2505253}) respectively, and $g_{1}(n,k,t,s)$, $g_{2}(n,t,s)$, $g_{3}(n,t,s)$, $g_{4}(n,k,t)$ are defined in (\ref{2505252}), (\ref{2505255}), (\ref{2505262}), (\ref{2505261}) respectively.  The following lemmas show some inequalities concerning those functions.

\begin{lem}\label{2502181}
	Let $n$, $k$, $t$, $s$ and $\ell$ be positive integers with $k\geq t+1$ and $  n\geq \ell(t+1)(2(k-t+1)^{2}+7s)$. Then the following hold.
	\begin{enumerate}[\normalfont(i)]
		\item  $f_{1}(n,k,t,s,x)>f_{1}(n,k,t,s,x+1)$ for any $x\in\left\{t,t+1,\ldots,k-1\right\}$. 
		\item  $\frac{6}{7\ell}f_{1}(n,k,t,s,k-1)>s\binom{k}{t}\binom{2k-2t+2}{k-t+1}$.
	\end{enumerate}
\end{lem}

\begin{proof}
	(i) For each $x\in\left\{t,t+1,\ldots,k-1\right\}$, we have 
	$f_{1}(n,k,t,s,x)\geq (k-t+1)^{x-t}\binom{x}{t}\binom{n-x}{k-x}$.
	It follows from $n-t\geq 2(t+1)(k-t+1)(k-t)$ and $n-k+1\geq 7s(t+1)$ that
	\begin{equation*}
		\begin{aligned}
			\frac{f_{1}(n,k,t,s,x+1)}{f_{1}(n,k,t,s,x)}\leq\ & \frac{ (k-t+1)^{x-t+1}\binom{x+1}{t}\binom{n-x-1}{k-x-1}+\sum_{i=0}^{x-t}s(k-t+1)^{i}\binom{x+1}{t}    }{(k-t+1)^{x-t}\binom{x}{t}\binom{n-x}{k-x}} \\
			=\ & \frac{(k-t+1)(x+1)(k-x)}{(x-t+1)(n-x)}+\frac{s((k-t+1)^{x-t+1}-1)(x+1)}{(k-t)(k-t+1)^{x-t}(x-t+1)\binom{n-x}{k-x}}\\
			\leq\ & (k-t+1)\cdot \frac{x+1}{x-t+1}\cdot \frac{k-x}{n-x}+\frac{k-t+1}{k-t}\cdot\frac{x+1}{x-t+1}\cdot\frac{s}{\binom{n-x}{k-x}}\\
			\leq\ & \frac{(t+1)(k-t+1)(k-t)}{n-t}+\left(1+\frac{1}{k-t}\right)\frac{s(t+1)}{n-k+1}<1.
		\end{aligned}
	\end{equation*}
	Then the required result follows.

	(ii) If $k=t+1$, then by $n-t>7s\ell (t+1)$, we have
	$$\frac{6}{7\ell}f_{1}(n,k,t,s,k-1)=\frac{6(n-t)}{7\ell}> 6s (t+1)= s\binom{k}{t}\binom{2k-2t+2}{k-t+1},$$
	as desired. Next assume $k\geq t+2$. 
	
	Since $f_{1}(n,k,t,s, k-1)\geq (k-t+1)^{k-t-1}\binom{k-1}{t}(n-k+1)$, we have 
	\begin{equation*}
		\begin{aligned}
			\frac{s\binom{k}{t}\binom{2k-2t+2}{k-t+1}}{f_{1}(n,k,t,s, k-1)}&\leq \frac{s\binom{k}{t}\binom{2k-2t+2}{k-t+1}}{(k-t+1)^{k-t-1}\binom{k-1}{t}(n-k+1)}=\frac{2s\left(1+\frac{t}{k-t}\right)\left(2+\frac{1}{k-t}\right)\binom{2k-2t}{k-t-1}}{(k-t+1)^{k-t-1}(n-k+1)}\\
			&\leq \frac{2s\left(1+\frac{t}{k-t}\right)\left(2+\frac{1}{k-t}\right)\left(\frac{k-t+3}{2}\right)^{k-t-2}(k-t+2)}{(k-t+1)^{k-t-1}(n-k+1)}\\
			&= \frac{2s\left(1+\frac{t}{k-t}\right)\left(2+\frac{1}{k-t}\right)\left(\frac{1}{2}+\frac{1}{k-t+1}\right)^{k-t-2}\left(1+\frac{1}{k-t+1}\right)}{n-k+1}\\
			&\leq \frac{20s\left(1+\frac{t}{2}\right)}{3(n-k+1)}\leq \frac{5s(t+1)}{n-k+1}.
		\end{aligned}
	\end{equation*}
	This together with $n-k+1\geq 7s\ell (t+1)$ yields the desired result. 
\end{proof}

\begin{lem}\label{2503174}
	Let $n$, $t$ and $s$ be positive integers with $n\geq 5s(t+1)^{2}$. Then the following hold.
	\begin{enumerate}[\normalfont(i)]
		\item $g_{2}(n,t,s)>(t+1)(n-t)-t$.
		\item $g_{3}(n,t,s)>\max \left\{2n,\left((t+1)(s+2)+s\right)^{2} \right\}$. 
		\item $g_{2}(n,t,s)>g_{3}(n,t,s)$ if $t\geq s+2$; 
		$g_{2}(n,t,s)<g_{3}(n,t,s)$ if $t\leq s+1$.
	\end{enumerate}
\end{lem}
\begin{proof}
	
	(i)  The desired result follows from $g_{2}(n,t,s)=(t+1)(n-t)+s-t$.
	
	(ii) Since $n\geq 5s(t+1)^{2}$, we have 
	$$g_{3}(n,t,s)-2n=sn-ts-2t+(s+2)^{2}\geq 5s^{2}(t+1)^{2}-ts-2t+(s+2)^{2}>0.$$
	We also conclude that
	\begin{equation*}
		\begin{aligned}
			g_{3}(n,t,s)&\geq 5s(s+2)(t+1)^{2}+(s+2)^{2}-t(s+2)\\
			&= 3s(s+2)(t+1)^{2}+s(s+2)(t+1)^{2}+(s+2)^{2}+s(s+2)(t+1)^{2}-t(s+2)\\
			&>(s+2)^{2}(t+1)^{2}+2s(s+2)(t+1)+s^{2}=\left((t+1)(s+2)+s\right)^{2}.
		\end{aligned}
	\end{equation*}
	Then the desired result holds.
	
	(iii) It is routine to check 
	$$g_{3}(n,t,s)-g_{2}(n,t,s)=(s-t+1)n+t^{2}-st+(s+2)^{2}-s.$$
	If $t\geq s+2$, then by $n\geq 5s(t+1)^{2}$, we have
	$$g_{3}(n,t,s)-g_{2}(n,t,s)\leq -5s(t+1)^{2}+2t^{2}<0.$$
	If $t\leq s+1$, then $$g_{3}(n,t,s)-g_{2}(n,t,s)\geq t^{2}-st+(s+2)^{2}-s\geq 2t(s+2)-st-s>0.$$ This finishes the proof.
\end{proof}

In the following, we always assume that $n$, $k$, $t$ and $s$ are positive integers with $k\geq t+2$ and $n\geq (t+1)^{2}(2(k-t+1)^{2}+7s)$. Then we have 
\begin{equation}\label{2503273}
	\binom{n-t-1}{k-t-1}\geq n-t-1\geq \max\left\{2(t+1)^{2}(k-t+1)^{2}, 7s(t+1)^{2}\right\}. 
\end{equation}
For $a\in\left\{0,1\right\}$, by	applying Pascal's formula repeatedly, we obtain
\begin{equation}\label{2506123}
	 \binom{n-t}{k-t}-\binom{n-k-a}{k-t}=\sum_{i=0}^{k+a-t-1}\binom{n-t-1-i}{k-t-1}\leq (k+a-t)\binom{n-t-1}{k-t-1}.
\end{equation}

\begin{lem}\label{2503091}
	The following hold.
	\begin{enumerate}[\normalfont(i)]
		\item $g_{1}(n,k,t,s)>f_{1}(n,k,t,s,t+1)^{2}$.
		\item $g_{1}(n,k,t,s)>f_{1}(n,k,t,s,t)f_{1}(n,k,t,s,t+2)$.
     	\item $g_{1}(n,k,t,s)>\frac{6}{7(t+1)}f_{1}(n,k,t,s,t+1)\left(\binom{n-t}{k-t}+s\right)$.
     	\item $g_{1}(n,k,t,s)>f_{2}(n,k,t,s)\left( \binom{n-t}{k-t}+t(k-t)\binom{n-t-1}{k-t-1}+s\right)$.
	\end{enumerate}
\end{lem}
\begin{proof}

	By (\ref{2506123}), we get 	$g_{1}(n,k,t,s)>\sum_{i=0}^{k-t}\binom{n-t-1-i}{k-t-1}\binom{n-t}{k-t}$.
	For each integer $i$ with $0\leq i\leq k-t$, by \cite[Proposition 1.6]{2504063}, we have 	
		$\frac{\binom{n-t-1-i}{k-t-1}}{\binom{n-t-1}{k-t-1}}\geq 1-\frac{i(k-t-1)}{n-t-1}$.
	It follows from (\ref{2503273})  that
	\begin{equation*}
		\begin{aligned}
			\frac{g_{1}(n,k,t,s)}{(k-t+1)\binom{n-t-1}{k-t-1}\binom{n-t}{k-t}}&> \sum_{i=0}^{k-t}\frac{\binom{n-t-1-i}{k-t-1}}{(k-t+1)\binom{n-t-1}{k-t-1}}\geq \sum_{i=0}^{k-t} \left( \frac{1}{k-t+1}-\frac{i(k-t-1)}{(k-t+1)(n-t-1)}\right)\\
			&= 1-\frac{(k-t)(k-t-1)}{2(n-t-1)}\geq 1-\frac{1}{4(t+1)^{2}}\geq \frac{15}{16},
		\end{aligned}
	\end{equation*}
	which implies
	\begin{equation}\label{2506121}
		g_{1}(n,k,t,s)>\frac{15}{16}(k-t+1)\binom{n-t-1}{k-t-1}\binom{n-t}{k-t}.
	\end{equation}

	(i) Applying (\ref{2503273}), we conclude 
	\begin{equation*}
		\begin{aligned}
			&\frac{f_{1}(n,k,t,s,t+1)^{2}}{(k-t+1)\binom{n-t-1}{k-t-1}\binom{n-t}{k-t}}=  \frac{k-t}{n-t}\left((t+1)^{2}(k-t+1)+\frac{2s(t+1)^{2}}{\binom{n-t-1}{k-t-1}} +\frac{s^{2}(t+1)^{2}}{(k-t+1)\binom{n-t-1}{k-t-1}^{2}}\right)\\ 
			\leq &\ \frac{1}{2}+\frac{1}{7(t+1)^{2}(k-t+1)}+\frac{1}{98(t+1)^{4}(k-t+1)^{2}}
			\leq  \frac{7225}{14112}.
		\end{aligned}
	\end{equation*}
   This together with	(\ref{2506121}) yields (i).

	(ii) From (\ref{2503273}), we obtain 
	\begin{equation*}
		\begin{aligned}
			&\frac{f_{1}(n,k,t,s,t)f_{1}(n,k,t,s,t+2)}{(k-t+1)\binom{n-t-1}{k-t-1}\binom{n-t}{k-t}}= \frac{\binom{t+2}{2}(k-t+1)(k-t-1)}{n-t-1}+\frac{s\binom{t+2}{2}}{\binom{n-t-1}{k-t-1}}\left(1+\frac{1}{k-t+1}\right)\\
			\leq &\  \frac{1}{4}\left(1+\frac{1}{t+1}\right)+\frac{1}{14}\left( 1+\frac{1}{k-t+1}\right)\left( 1+\frac{1}{t+1}\right)
			\leq \frac{29}{56}.
		\end{aligned}
	\end{equation*}	
	It follows from	(\ref{2506121}) that (ii) holds.

	(iii) According to (\ref{2503273}), we get
	\begin{equation*}
		\begin{aligned}
			&\frac{	6f_{1}(n,k,t,s,t+1)\left(\binom{n-t}{k-t}+s\right)}{7(t+1)(k-t+1)\binom{n-t-1}{k-t-1}\binom{n-t}{k-t}}=  \frac{6}{7}\left( 1+\frac{s}{(k-t+1)\binom{n-t-1}{k-t-1}}\right)\left(1+\frac{s(k-t)}{(n-t)\binom{n-t-1}{k-t-1}}\right) \\
			\leq&\ \frac{6}{7}\left( 1+\frac{1}{7(t+1)^{2}(k-t+1)}\right)\left( 1+\frac{1}{14(t+1)^{4}(k-t+1)}\right)\leq \frac{57205}{65856}.
		\end{aligned}
	\end{equation*}	
	From	(\ref{2506121}), we know that (iii) holds.

	(iv) According to	(\ref{2503273}), we have	
	  \begin{equation*}
		\begin{aligned}
			&\ \frac{f_{2}(n,k,t,s)\left( \binom{n-t}{k-t}+t(k-t)\binom{n-t-1}{k-t-1}+s\right)}{(k-t+1)\binom{n-t-1}{k-t-1}\binom{n-t}{k-t}}\\
			=& \left( \frac{1}{k-t+1}+\frac{(k-t)(k-t-1)}{n-t-1}+\frac{s}{\binom{n-t-1}{k-t-1}}\right)\left(1+\frac{t(k-t)^{2}}{n-t}+\frac{s(k-t)}{(n-t)\binom{n-t-1}{k-t-1}}\right)\\
			\leq & \left( \frac{1}{k-t+1}+\frac{1}{2(t+1)^{2}}+\frac{1}{7(t+1)^{2}}\right)\left(1+\frac{1}{2(t+1)}+\frac{1}{14(t+1)^{4}(k-t+1)}\right)
			\leq  \frac{69803}{112896}.
		\end{aligned}
	\end{equation*}
	By (\ref{2506121}), we get that  (iv) follows.
\end{proof}

\begin{lem}\label{2503042}
	The following hold.
	\begin{enumerate}[\normalfont(i)]
		\item $f_{3}(n,k,t,s,k)\geq f_{3}(n,k,t,s,x)$ for any $x\in\left\{t+2,t+3,\ldots, k\right\}$.
		\item $f_{3}(n,k,t,s,k)\left( \binom{n-t}{k-t}+t\left(k-t\right)\binom{n-t-2}{k-t-2}+s\right)<g_{1}(n,k,t,s)$.
	\end{enumerate}
\end{lem}
\begin{proof}
	(i)  If $k=t+2$, then the required result follows. Next we assume $k\geq t+3$.

	By \cite[Proposition 1.6]{2504063}, we have $\binom{n-k-1}{k-t-2}\geq \left(1-\frac{(k-t-1)(k-t-2)}{n-t-2}\right)\binom{n-t-2}{k-t-2}$. It follows from $n-t-2\geq 2(t+1)^{2}(k-t+1)^{2}\geq 4(t+1)(k-t+1)k$ that for each $x\in\left\{t+2,t+3,\ldots, k-1\right\}$,
	\begin{equation*}
		\begin{aligned}
			&f_{3}(n,k,t,s,x+1)-f_{3}(n,k,t,s,x)\\
			=&\  \frac{n-x-1}{k-t-1}\binom{n-x-2}{k-t-2}-(2k-2x+1)\binom{n-t-2}{k-t-2}\\
			\geq&\ \frac{n-k}{k-t-1}\binom{n-k-1}{k-t-2}-(2k-2t-3)\binom{n-t-2}{k-t-2}\\
			\geq& \left(  \frac{n-k}{k-t-1}\left(1-\frac{(k-t-1)(k-t-2)}{n-t-2}\right)-2(k-t+1) \right)\binom{n-t-2}{k-t-2}\\
			\geq& \left( \frac{3(t+1)(k-t+1)k}{k-t-1}\left( 1-\frac{(k-t-1)(k-t-2)}{2(t+1)^{2}(k-t+1)^{2}}\right)-2(k-t+1)     \right)\binom{n-t-2}{k-t-2}\\
			\geq& \left( \frac{21}{8}(t+1)(k-t+1)-2(k-t+1)\right)\binom{n-t-2}{k-t-2}>0.
		\end{aligned}
	\end{equation*}
	Then the desired result holds.
	
	(ii) By (\ref{2506123}), we have $\binom{n-t}{k-t}-\binom{n-k}{k-t}\leq (k-t)\binom{n-t-1}{k-t-1}$. It follows from (\ref{2503273}) that
	\begin{equation}\label{2504065}
		\begin{aligned}
			&\frac{f_{3}(n,k,t,s,k)}{\binom{n-t}{k-t}}\leq \frac{(k-t)^{2}}{n-t}+\frac{k-t}{n-t}\cdot\frac{k-t-1}{n-t-1}+\frac{2s(k-t)}{(n-t)\binom{n-t-1}{k-t-1}}\\
			\leq &\ \frac{1}{2(t+1)^{2}}+\frac{1}{4(t+1)^{4}(k-t+1)^{2}}+\frac{1}{7(t+1)^{4}(k-t+1)}
			\leq \frac{523}{4032}.
		\end{aligned}
	\end{equation}
	
	For convenience, write
	$$h(n,k,t,s)=\binom{n-t-2}{k-t-2}\binom{n-t}{k-t}+2s\binom{n-t}{k-t}+f_{3}(n,k,t,s,k)\left( t(k-t)\binom{n-t-2}{k-t-2}+s\right).$$
	From (\ref{2503273}) and (\ref{2504065}), we obtain
	\begin{equation}\label{2505227}
		\begin{aligned}
			&\frac{h(n,k,t,s)}{\binom{n-t-1}{k-t-1}\binom{n-t}{k-t}}=\frac{k-t-1}{n-t-1}+\frac{2s}{\binom{n-t-1}{k-t-1}}+\frac{f_{3}(n,k,t,s,k)}{\binom{n-t}{k-t}}\left( \frac{t(k-t)(k-t-1)}{n-t-1}+\frac{s}{\binom{n-t-1}{k-t-1}}\right)\\
			\leq&\ \frac{1}{2(t+1)^{2}(k-t+1)}+\frac{2}{7(t+1)^{2}}+\frac{523}{4032}\left( \frac{1}{2(t+1)}+\frac{1}{7(t+1)^{2}}\right)
			\leq \frac{2119}{14112}.
		\end{aligned}
	\end{equation}
	
	By \cite[Proposition 1.6]{2504063}, we have $\binom{n-k-1}{k-t-1}\geq \left(1-\frac{(k-t)(k-t-1)}{n-t-1}\right)\binom{n-t-1}{k-t-1}$. It follows from (\ref{2503273}) that
	\begin{equation*}
		\begin{aligned}
			\frac{g_{1}(n,k,t,s)-\left(\binom{n-t}{k-t}-\binom{n-k}{k-t}\right)\binom{n-t}{k-t}}{\binom{n-t-1}{k-t-1}\binom{n-t}{k-t}}&\geq \frac{\binom{n-k}{k-t}-\binom{n-k-1}{k-t}}{\binom{n-t-1}{k-t-1}}=\frac{\binom{n-k-1}{k-t-1}}{\binom{n-t-1}{k-t-1}}
		\geq 1- \frac{1}{2(t+1)^{2}}
			\geq \frac{7}{8}.
		\end{aligned}
	\end{equation*}
	This combining with  (\ref{2505227}) yields
	\begin{equation*}
		\begin{aligned}
			&\ g_{1}(n,k,t,s)-f_{3}(n,k,t,s,k)\left( \binom{n-t}{k-t}+t\left(k-t\right)\binom{n-t-2}{k-t-2}+s\right)\\
			=&\ g_{1}(n,k,t,s)- \left(\binom{n-t}{k-t}-\binom{n-k}{k-t}\right)\binom{n-t}{k-t}-h(n,k,t,s)\\
			\geq& \left( \frac{7}{8}-\frac{2119}{14112}\right)\binom{n-t-1}{k-t-1}\binom{n-t}{k-t}>0.
		\end{aligned}
	\end{equation*}
	This finishes the proof of (ii).
\end{proof}

\begin{lem}\label{2504143}
		The following hold.
	\begin{enumerate}[\normalfont(i)]
		\item  $g_{1}(n,k,t,s)>\left(\binom{n-t}{k-t}-\binom{n-k-1}{k-t}\right)\left( \binom{n-t}{k-t}+t\right)$.
		\item  If $k\leq 2t$ and $(k,t)\neq (4,2)$, then $g_{1}(n,k,t,s)<g_{4}(n,k,t)$.
	\end{enumerate}
\end{lem}

\begin{proof}

(i) By (\ref{2506123}), we get $\binom{n-t}{k-t}-\binom{n-k-1}{k-t}\leq (k-t+1)\binom{n-t-1}{k-t-1}$. It follows from (\ref{2503273})  that
\begin{equation*}
	\begin{aligned}
		&g_{1}(n,k,t,s)-\left(\binom{n-t}{k-t}-\binom{n-k-1}{k-t}\right)\left( \binom{n-t}{k-t}+t\right)\\
		\geq &\ s\binom{n-t}{k-t}- t\left( \binom{n-t}{k-t}-\binom{n-k-1}{k-t}\right)
		\geq   \left( \frac{s(n-t)}{k-t}-t(k-t+1)\right)\binom{n-t-1}{k-t-1}\\
		\geq &\left( \frac{2s(t+1)^{2}(k-t+1)^{2}}{k-t}-t(k-t+1)   \right)\binom{n-t-1}{k-t-1}>0.
	\end{aligned}
\end{equation*}	
Then (i) holds.

(ii) We divide our proof into the following cases.

\medskip
\noindent \textbf{Case 1.} $k<2t$. 
\medskip

According to (\ref{2503273}), we obtain 
\begin{equation}\label{2504146}
\begin{aligned}	    \binom{n-t-1}{k-t-1}\left(\binom{n-t-1}{k-t-1}-st\right)-s^{2}&\geq  7s(t+1)^{2}\left( 7s(t+1)^{2}-st\right)-s^{2}>0,\\
\left( 1-\frac{(t+1)(k-t)}{n-t}   \right)\binom{n-t-1}{k-t-1}-s&\geq 7s(t+1)^{2}\left(1-\frac{1}{2(t+1)(k-t+1)} \right)-s>0.
	\end{aligned}
 \end{equation}

	From  $k\leq 2t-1$ and  (\ref{2506123}), we get $\binom{n-t}{k-t}-\binom{n-k-1}{k-t}\leq t\binom{n-t-1}{k-t-1}$. It follows that 
$$g_{1}(n,k,t,s)\leq t\binom{n-t-1}{k-t-1}\binom{n-t}{k-t}+st\binom{n-t-1}{k-t-1}+s\binom{n-t}{k-t}+s^{2}=:m(n,k,t,s). $$
Therefore, we have
\begin{equation*}
	\begin{aligned}
		&\ g_{4}(n,k,t)-g_{1}(n,k,t,s)\geq g_{4}(n,k,t)-m(n,k,t,s)\\
		=& \left( \left( 1-\frac{(t+1)(k-t)}{n-t}   \right)\binom{n-t-1}{k-t-1}-s\right)\binom{n-t}{k-t}+\binom{n-t-1}{k-t-1}\left(\binom{n-t-1}{k-t-1}-st \right)-s^{2}.
	\end{aligned}
\end{equation*}
This together with (\ref{2504146}) 	 yields the desired result.

\medskip
\noindent \textbf{Case 2.} $k=2t$. 
\medskip

Recall that $k\geq t+2$  and $(k,t)\neq (4,2)$, which implies $t\geq 3$. By (\ref{2506123}), we have $\binom{n-t}{t}-\binom{n-2t-1}{t}\leq (t+1)\binom{n-t-1}{t-1}$. It follows from  (\ref{2503273})  that
\begin{equation*}
	\begin{aligned}
		&\ \frac{g_{1}(n,2t,t,s)-\left( \binom{n-t}{t}-\binom{n-2t-1}{t}\right)\binom{n-t}{t}}{\binom{n-t-1}{t-1}^{2}}
		\leq \frac{s(t+1)\binom{n-t-1}{t-1}+s\binom{n-t}{t}+s^{2}}{\binom{n-t-1}{t-1}^{2}}\\
		=&\  \frac{s(t+1)}{\binom{n-t-1}{t-1}}+	\frac{t-1}{t}\cdot\left(1+\frac{2}{n-t-2} \right)\cdot\frac{s(t-2)}{n-t-1}\cdot\frac{1}{\binom{n-t-3}{t-3}}+\frac{s^{2}}{\binom{n-t-1}{t-1}^{2}}\\		
	  \leq&\ \frac{1}{7(t+1)}+\frac{2}{7(t+1)}+\frac{1}{49(t+1)^{4}}<1.
	\end{aligned}
\end{equation*}
Note that $g_{4}(n,2t,t)=(t+1)\binom{n-t-1}{t-1}\binom{n-t-1}{t}+\binom{n-t-1}{t-1}^{2}$. To prove $g_{1}(n,2t,t,s)<g_{4}(n,2t,t)$, it is sufficient to show 
\begin{equation}\label{2504148}
\left( \binom{n-t}{t}-\binom{n-2t-1}{t}\right)\binom{n-t}{t}<	(t+1)\binom{n-t-1}{t-1}\binom{n-t-1}{t}.
\end{equation}

\medskip
\noindent \textbf{Case 2.1.} $t=3$. 
\medskip

It is routine to check that 
$$4\binom{n-4}{2}\binom{n-4}{3}-\left( \binom{n-3}{3}-\binom{n-7}{3}\right)\binom{n-3}{3}=\frac{(n-9)(n-5)(n-4)}{3}.$$
This together with $n\geq 512+112s$ yields (\ref{2504148}).

\medskip
\noindent \textbf{Case 2.2.} $t\geq 4$. 
\medskip

For each integer $i$ with $0\leq i\leq t$, we claim that 
\begin{equation}\label{2504151}
	\frac{\binom{n-t-1-i}{t-1}}{\binom{n-t-1}{t-1}}\leq \frac{n-2t}{n-t-1+(i-1)(t-1)}. 
\end{equation}
If $i=0$ or $i=1$, then the required result holds. If $2\leq i\leq t$, then by $ba\geq (b+1)(a-1)$ for any $b\geq a$, we have 
$$(n-t-1)(n-t-2)\cdots (n-t-i)\geq \left( n-t-1+(i-1)(t-1)\right)(n-2t-1)\cdots(n-2t-(i-1)). $$
This together with
$$\frac{\binom{n-t-1-i}{t-1}}{\binom{n-t-1}{t-1}}=\frac{(n-2t)(n-2t-1)\cdots(n-2t-(i-1))}{(n-t-1)(n-t-2)\cdots(n-t-i)}$$yields (\ref{2504151}).

From (\ref{2506123}), we obtain
$\binom{n-t}{t}-\binom{n-2t-1}{t}=\sum_{i=0}^{t}\binom{n-t-1-i}{t-1}$. It follows from (\ref{2504151}) that
\begin{equation}\label{2504161}
	\begin{aligned}
		\frac{\left(   \binom{n-t}{t}-\binom{n-2t-1}{t}\right)\binom{n-t}{t}}{\binom{n-t-1}{t-1}\binom{n-t-1}{t}}&=\frac{n-t}{n-2t}\cdot\sum_{i=0}^{t}\frac{\binom{n-t-1-i}{t-1}}{\binom{n-t-1}{t-1}}\leq \sum_{i=0}^{t}\frac{n-t}{n-t-1+(i-1)(t-1)}\\
		&\leq \frac{n-t}{n-2t}+\frac{n-t}{n-t-1}+3\cdot\frac{n-t}{n-2}+t-4\\
		&= \frac{n^{2}(7-2t)+4n(2t^{2}-5t-2)+2t(-3t^{2}+4t+9)}{(n-2)(n-2t)(n-t-1)}    +t+1.
	\end{aligned}
\end{equation}

Since $t\geq 4$ and $n\geq 2(t+1)^{4}$, we obtain
$$n^{2}(7-2t)+4n(2t^{2}-5t-2)\leq -2(t+1)^{4}n+8t^{2}n<0,\ -3t^{2}+4t+9<0.$$
This together with (\ref{2504161}) yields  (\ref{2504148}). Hence  (ii) holds.
\end{proof}

\medskip
\noindent{\bf Acknowledgment.}	
 K. Wang is supported by the National Key R\&D Program of China (No. 2020YFA0712900), National Natural Science Foundation of China (12071039, 12131011) and  Beijing Natural Science Foundation (1252010). T. Yao is supported by the Henan province Natural Science Foundation (252300420899).

\end{document}